\documentclass{amsart}

\usepackage{amssymb,latexsym}
\usepackage{graphicx}
\usepackage[all,2cell]{xy}
\SelectTips{eu}{12}

\newcommand{\Z}{\mathbb{Z}}

\newcommand{\Q}{\mathbb{Q}}
\newcommand{\R}{\mathbb{R}}
\newcommand{\C}{\mathbb{C}}
\newcommand{\Ha}{\mathbb{H}}

\newtheorem{lemma}{Lemma}[section]

\newtheorem{proposition}[lemma]{Proposition}
\newtheorem{theorem}[lemma]{Theorem}
\newtheorem{corollary}[lemma]{Corollary}

\theoremstyle{definition}
\newtheorem{remark}[lemma]{Remark}
\newtheorem{definition}[lemma]{Definition}
\newtheorem{example}[lemma]{Example}

\begin{document}
\parindent0em
\setlength\parskip{.1cm}
\title[Homology of moduli spaces of linkages]{Homology of moduli spaces of linkages in high-dimensional Euclidean space}
\author{Dirk Sch\"utz}
\address{Department of Mathematical Sciences\\ University of Durham\\ Science Laboritories\\ South Rd\\ Durham DH1 3LE\\ United Kingdom}
\email{dirk.schuetz@durham.ac.uk}
\begin{abstract}
 We study the topology of moduli spaces of closed linkages in $\R^d$ depending on a length vector $\ell\in \R^n$. In particular, we use equivariant Morse theory to obtain information on the homology groups of these spaces, which works best for odd $d$.  In the case $d=5$ we calculate the Poincar\'e polynomial in terms of combinatorial information encoded in the length vector.
\end{abstract}
\maketitle

\section{Introduction}

In this paper, we consider polygons, or linkages, with fixed side lengths in Euclidean space $\R^d$. The topology of the corresponding moduli spaces $\mathcal{M}_d(\ell)$, see Section \ref{basicdef} for precise definitions, has been studied extensively in the cases $d=2$ and $3$. In particular, a lot of information on the homology and cohomology of these spaces have been obtained, see for example \cite{farsch,hauknu,kateto,klyach,walker}. Furthermore, cohomology can be used to show that the topological type of $\mathcal{M}_d(\ell)$ for $d=2$ and $3$ is determined by the length vector $\ell\in \R^n$, \cite{fahash,schuet}.

A lot less is known for $d>3$. Kamiyama \cite{kamiya} has obtained a formula for the Euler characteristic of $\mathcal{M}_4(\ell)$ in the equilateral case, that is, when $\ell=(1,1,\ldots,1)\in \R^n$. Schoenberg \cite{schoen} shows that $\mathcal{M}_d(\ell)$ is homeomorphic to a disc if $d\geq n$, which implies that it is homeomorphic to a sphere for $d=n-1$.

Another class of examples which has been extensively studied is given by $\ell=(1,\ldots,1,n-2)\in \R^n$, as in this case $\mathcal{M}_d(\ell)$ coincides with the shape space $\Sigma^{n-1}_{d-1}$, which is thoroughly examined in \cite{kebacl}. In particular, the homology groups of these spaces are completely known \cite{kebacl}. It follows from their calculations that for $4\leq d < n-1$ the space $\Sigma^{n-1}_{d-1}$ is not a manifold.

Our method to study $\mathcal{M}_d(\ell)$ is through equivariant Morse theory. We use the fact that $\mathcal{M}_d(\ell)=\mathcal{C}_d(\ell)/SO(d-1)$, where $\mathcal{C}_d(\ell)$ is the so-called chain space, and construct a $SO(d-1)$-invariant Morse-Bott function on $\mathcal{C}_d(\ell)$. By analyzing the critical manifolds we obtain information on the homology of $\mathcal{M}_d(\ell)$. Our method works best for odd $d$ and rational coefficients, as the spectral sequence arising from the filtration given by the Morse-Bott function collapses. For even $d$ we can still obtain useful information on the topology of $\mathcal{M}_d(\ell)$.

\begin{theorem}\label{thmA}
 Let $\ell\in \R^n$ be a generic length vector such that $\mathcal{M}_d(\ell)\not=\emptyset$ for $d\geq 2$.
\begin{enumerate}
 \item For $d\geq 3$ the space $\mathcal{M}_d(\ell)$ is $((d-1)(d-2)/2 + d-3)$-connected.
 \item For $n\geq 3$, $\mathcal{M}_{n-1}(\ell)$ is homeomorphic to the sphere of dimension $n(n-3)/2$.
 \item For $4\leq d\leq n-2$, $\mathcal{M}_d(\ell)$ is not homotopy equivalent to a closed manifold.
\end{enumerate}

\end{theorem}

Our Morse-theoretic methods imply homotopy equivalence in (2), but using the result of Schoenberg \cite{schoen} the homeomorphism can be obtained directly.

Part (3) is obtained by showing that $\mathcal{M}_d(\ell)$ does not satisfy Poincar\'e duality. Note that for $d=2,3$ the space $\mathcal{M}_d(\ell)$ is a smooth closed manifold.

As we have mentioned before, our homology calculations work best for odd $d$. To obtain simple formulas for the Poincar\'e polynomial, we stick to the case $d=5$. Define
\begin{eqnarray*}
 Q_{2m}(t)&=&\frac{(t^{m+1}-1)^2}{(t-1)^2}\\
Q_{2m+1}(t)&=& \frac{(t^{m+2}-1)(t^{m+1}-1)}{(t-1)^2}
\end{eqnarray*}
for all $m\geq 0$.

\begin{theorem}\label{thmB}
 Let $\ell\in \R^n$ be a generic length vector. Let $n\geq 5$ and $m=\lfloor\frac{n+1}{2}\rfloor$. Then there exist non-negative integers $c_i(\ell)$ depending only on $\ell$ such that the Poincar\'e polynomial of $\mathcal{M}_5(\ell)$ is
\begin{eqnarray*}
P_5^\ell(t)&=&1+ t^9\cdot\, \sum_{i=0}^{m-2}(c_i(\ell)\,(Q_{n-6-i}(t^4)-Q_{i-4}(t^4)).
\end{eqnarray*}
\end{theorem}
The exact form of $c_i(\ell)$ can be seen in Theorem \ref{polynomiald=5}.

The case $d=5$ also contains interesting geometry closely related to the case $d=3$. Indeed, $\mathcal{M}_3(\ell)$ carries extra symplectic and K\"ahler structures, which have been studied in detail in \cite{hauknu,kapmil,klyach}. In particular, in \cite{kapmil} a complex analytic equivalence between $\mathcal{M}_3(\ell)$ and a weighted quotient of $(S^2)^n$ by $PSL(2,\C)$ is established. Foth and Lozano \cite{fotloz} obtain an analogous statement for $\mathcal{M}_5(\ell)$ and a weighted quotient of $(S^4)^n$ by $PSL(2,\Ha)$. They also generalize the Gel'fand-MacPherson correspondence to the quaternion context and realize $\mathcal{M}_5(\ell)$ as a quotient of a subspace in a quaternion Grassmannian.

It can be easily read off from Theorem \ref{thmB} that the reduced rational homology of $\mathcal{M}_5(\ell)$ starts in degree $9$ and is limited to odd degrees. In particular the rational cohomology ring structure is trivial. Given that the cohomology ring structure is instrumental in distinguishing topological types for $d=2$ and $3$, one would hope for more algebraic information also in the cases $d\geq 4$. A suitable setting for this appears to be intersection homology, which we plan to examine in a future project.

The paper is organized as follows. Section \ref{basicdef} collects some basic properties of linkage spaces and Section \ref{morsebott} introduces the equivariant Morse-Bott function. In Section \ref{perfectthree}  we recover some well known results from \cite{klyach} and \cite{hauknu} on the homology of $\mathcal{M}_3(\ell)$. A cell decomposition for $\mathcal{M}_d(\ell)$ based on the Morse-Bott function is obtained in Sections \ref{sec6} and \ref{boundarysection}, which is used to prove Theorem \ref{thmA}. Local homology calculations are done in Sections \ref{sec8} and \ref{sec9}, which culminate in the proof of Theorem \ref{thmB} in Section \ref{sec10}. We also obtain some Euler characteristic results for even $d$ in Section \ref{eulersec}. There are two appendices, one showing the equivalence of the shape space with a certain linkage space, and one deals with basic properties of the polynomials $Q_n(t)$.

\section{Basic definitions and properties of linkage spaces}
\label{basicdef}

Let $d,n$ be positive integers and $\ell=(\ell_1,\ldots,\ell_n)$ satisfy $0<\ell_i$ for all $i=1,\ldots,n$. We call $\ell$ a \em length vector\em. The \em moduli space of $\ell$ \em in $\R^d$ consists of all closed linkages with lengths $\ell$ up to rotations and translations. We can describe this space as
\begin{eqnarray*}
 \mathcal{M}_d(\ell)&=&\left\{(z_1,\ldots,z_n)\in (S^{d-1})^n\,\left|\,\sum_{i=1}^n \ell_iz_i=0\right\}\right/SO(d)
\end{eqnarray*}
where $SO(d)$ acts diagonally on $(S^{d-1})^n$. We also denote the \em space of chains \em of $\ell$ as
\begin{eqnarray*}
 \mathcal{C}_d(\ell)&=&\left\{(z_1,\ldots,z_{n-1})\in (S^{d-1})^{n-1}\,\left|\, \sum_{i=1}^{n-1} \ell_iz_i=(-\ell_n,0,\ldots,0)\right\}\right.
\end{eqnarray*}
If we let $SO(d-1)$ act on $S^{d-1}$ by fixing the first coordinate, we get an $SO(d-1)$-action on $\mathcal{C}_d(\ell)$ such that
\begin{eqnarray*}
 \mathcal{M}_d(\ell)&=& \mathcal{C}_d(\ell)/SO(d-1).
\end{eqnarray*}

It is clear that permuting the coordinates of $\ell$ does not change the homeomorphism type of $\mathcal{M}_d(\ell)$. However this is not true for $\mathcal{C}_d(\ell)$. In the cases $d=1,2,4,8$ one can use the multiplication structure of $S^{d-1}$ to construct a homeomorphism $\mathcal{C}_d(\ell)\cong \mathcal{C}_d(\sigma\ell)$ for any permutation $\sigma$, but for other values of $d$ these spaces are usually not homeomorphic, compare \cite[Rm.2.2]{fahasc}.

\begin{definition}
 Let $\ell$ be a length vector. A subset $J\subset \{1,\ldots,n\}$ is called \em $\ell$-short\em, if
\begin{eqnarray*}
 \sum_{j\in J}\ell_j&<& \sum_{i\notin J}\ell_i.
\end{eqnarray*}
It is called \em $\ell$-long\em, if the complement is $\ell$-short, and \em $\ell$-median\em, if it is neither $\ell$-short nor $\ell$-long. The length vector is called \em generic\em, if there are no $\ell$-median subsets.

We also write
\begin{eqnarray*}
 \ell_J&=&\sum_{j\in J}\ell_j.
\end{eqnarray*}

For $m\in \{1,\ldots,n\}$ the length vector is called \em $m$-dominated\em, if $\ell_m\geq \ell_i$ for all $i=1,\ldots,n$.
\end{definition}

If the length vector is generic, there do not exist collinear configurations, that is, points $[z_1,\ldots,z_n]\in \mathcal{M}_d(\ell)$ for which all $z_i\in \{\pm x\}$ for some $x\in S^{d-1}$. Notice that generic is equivalent to $\mathcal{M}_1(\ell)=\emptyset$.

In the case that $\ell$ is generic, it is easy to see that $\mathcal{C}_d(\ell)$ is a closed manifold of dimension $(n-2)(d-1)-1$. In the case that $d=2$ or $d=3$, we then get that $SO(d-1)$ acts freely on $\mathcal{C}_d(\ell)$, and $\mathcal{M}_d(\ell)$ is also a closed manifold of dimension $(d-1)(n-3)$. For $d\geq 4$, the action is no longer free, and we will see that generally $\mathcal{M}_d(\ell)$ is not a manifold.

\begin{definition}
 Let $\ell\in \R^n$ be an $m$-dominated generic length vector. For $k\in \{0,\ldots,n-3\}$ we write
\begin{eqnarray*}
 \mathcal{S}^m_k(\ell)&=&\{J\subset \{1,\ldots,n\}\,|\,m\in J, |J|=k+1, J \mbox{ is }\ell\mbox{-short}\}.
\end{eqnarray*}
and
\begin{eqnarray*}
 a_k(\ell)&=&|\mathcal{S}^m_k(\ell)|.
\end{eqnarray*}

\end{definition}

So the union $\mathcal{S}^m_\ast(\ell)$ of $\mathcal{S}^m_k(\ell)$ over all $k=0,\ldots,n-3$ contains all short subsets $J\subset \{1,\ldots,n\}$  which include $m$. It is worth pointing out that $\mathcal{S}^m_\ast(\ell)$ is an abstract simplicial complex with 0-simplices given by $\mathcal{S}^m_1(\ell)$.

Note that a length vector can be $m$-dominated by more than one $m\in \{1,\ldots,n\}$. The numbers $a_k(\ell)$ however do not depend on this. We have $a_0(\ell)\leq 1$, and for a generic length vector it is easy to see that $\mathcal{M}_d(\ell)\not=\emptyset$ for $d\geq 2$ if and only if $a_0(\ell)=1$.

If $J\subset\{1,\ldots,n\}$, we define the hyperplane
\begin{eqnarray*}
 H_J&=&\left\{(x_1,\ldots,x_n)\in \R^n\,|\, \sum_{j\in J}x_j= \sum_{j\notin J}x_j\right\}
\end{eqnarray*}
and let
\begin{eqnarray*}
 \mathcal{H}&=& \R^n_{>0} - \bigcup_{J\subset\{1,\ldots,n\}}H_J,
\end{eqnarray*}
where $\R^n_{>0}=\{(x_1,\ldots,x_n)\in \R^n\,|\,x_i>0\}$. Then $\mathcal{H}$ has finitely many components, which we call \em chambers\em. It is clear that a length vector $\ell$ is generic if and only if $\ell\in\mathcal{H}$. 

It is shown in \cite{hausma} that if $\ell$ and $\ell'$ are in the same chamber, then $\mathcal{C}_d(\ell)$ and $\mathcal{C}_d(\ell')$ are $O(d-1)$-equivariantly diffeomorphic. In particular, $\mathcal{M}_d(\ell)$ and $\mathcal{M}_d(\ell')$ are homeomorphic.

It is easy to see that two $m$-dominated generic length vectors $\ell$, $\ell'$ are in the same chamber if and only if $\mathcal{S}_\ast^m(\ell)=\mathcal{S}_\ast^m(\ell')$.

\begin{definition}
 Let $z=(z_1,\ldots,z_{n-1})\in \mathcal{C}_d(\ell)$. The \em rank \em of $z$ is the maximal number of linearly independent vectors $z_1,\ldots,z_{n-1}\in \R^d$. Note that the rank remains the same under the $SO(d-1)$ action, and we can define the \em rank \em of $z=[z_1,\ldots,z_n]\in \mathcal{M}_d(\ell)$ also as the maximal number of linearly independent vectors $z_1,\ldots,z_n\in \R^d$.
\end{definition}

The natural inclusion $i:\mathcal{C}_{d-1}(\ell)\to \mathcal{C}_d(\ell)$ induces a natural map
\[
 \varphi:\mathcal{M}_{d-1}(\ell)\to \mathcal{M}_d(\ell).
\]
This map need not be injective: in fact, if $\ell=(1,1,1)$, it is clear that $\mathcal{M}_2(\ell)=S^0$ and $\mathcal{M}_d(\ell)=\{\ast\}$ for $d\geq 3$.

\begin{lemma}\label{easylemma}
 Let $\ell\in \R^n$ be a length vector.
\begin{enumerate}
 \item Let $n\leq d$. Then $\varphi:\mathcal{M}_{d-1}(\ell)\to \mathcal{M}_d(\ell)$ is surjective.
 \item Let $n\leq d-1$. Then $\varphi:\mathcal{M}_{d-1}(\ell)\to \mathcal{M}_d(\ell)$ is a homeomorphism.
 \item Let $z=(z_1,\ldots,z_{n-1})\in \mathcal{C}_d(\ell)$ satisfy ${\rm rank}\,z\geq d-1$. Then $z$ is only fixed by the identity element of $SO(d-1)$.
\end{enumerate}

\end{lemma}

\begin{proof}
 Let $n\leq d$ and $z=[z_1,\ldots,z_n]\in \mathcal{M}_d(\ell)$. Since $\sum\ell_iz_i=0$, the rank of $z$ is at most $n-1< d$. Thus there exists a $A\in SO(d)$ with $Az_i\in \R^{d-1}\times\{0\}\subset \R^d$ for all $i=1,\ldots,n$. But clearly $[Az_1,\ldots,Az_n]$ is in the image of $\varphi$.

If $n\leq d-1$ and $\varphi(z)=\varphi(z')$, we get ${\rm rank}\, z={\rm rank}\, z'\leq d-2$. After using rotations in $\R^{d-1}$ we can therefore assume that all $z_i,z_i'\in \R^{{\rm rank}\,z}\subset\R^{d-2}\subset \R^d$. By assumption there is $A\in SO(d)$ with $Az_i=z_i'$ for all $i=1,\ldots,n$, which therefore fixes $\R^{{\rm rank}\,z}\subset \R^{d-2}$. We can now extend $A|\R^{{\rm rank}\,z}\in O({\rm rank}\,z)$ to $B\in SO(d-1)$ with $Bz_i=z_i'$ for all $i=1,\ldots,n$. But this means $z=z'\in \mathcal{M}_{d-1}(\ell)$, and $\varphi$ is bijective, hence a homeomorphism by compactness.

Finally, if $z$ has rank at least $d-1$ and $Az=z$ with $A\in SO(d)$, choose a basis of $\R^d$ where the first $d-1$ elements are taken from the coordinates of $z$. Then $A$ fixes at least $d-1$ elements of a basis of $\R^d$ and is therefore the identity.
\end{proof}

We remark that if $z\in \mathcal{C}_d(\ell)$ satisfies ${\rm rank}\,z\geq d-1$, then $n\geq d$. One checks that for $n\geq d$ we can also find $z\in \mathcal{C}_d(\ell)$ which have ${\rm rank}\,z\geq d-1$. If we denote the dimension of $\mathcal{M}_d(\ell)$ by $\mathbf{d}^n_d$, we thus get for $n\geq d$ that
\begin{eqnarray*}
 \mathbf{d}^n_d&=&(n-3)(d-1)-\frac{(d-2)(d-3)}{2}.
\end{eqnarray*}

\section{A Morse-Bott function on the space of chains}
\label{morsebott}

In this section we will assume that $\ell\in \R^n$ is $n$-dominated.

Define the map $F:\mathcal{C}_d(\ell)\to \R$ by
\begin{eqnarray*}
 F(z_1,\ldots,z_{n-1})&=&\ell_{n-1} p_1(z_{n-1})+\ell_n
\end{eqnarray*}
where $p_1:\R^d\to \R$ is projection to the first coordinate. Notice that $F$ is $SO(d-1)$-invariant and
\begin{eqnarray*}
 F(z_1,\ldots,z_{n-1})&=&-p_1\left(\sum_{i=1}^{n-2}\ell_iz_i\right).
\end{eqnarray*}
We have obvious maxima and minima for points with $z_{n-1}=\pm e_1$. This leads to embeddings of $\mathcal{C}_d(\ell^+)$ and $\mathcal{C}_d(\ell^-)$ into $\mathcal{C}_d(\ell)$, where
\begin{eqnarray*}
 \ell^+&=&(\ell_1,\ldots,\ell_{n-2},\ell_n+\ell_{n-1})\\
\ell^-&=&(\ell_1,\ldots,\ell_{n-2},\ell_n-\ell_{n-1}).
\end{eqnarray*}
Note that for generic $\ell$ we can assume that $\ell_n>\ell_{n-1}$, but $\ell^-$ need not be $n-1$-dominated.

Let $J\subset \{1,\ldots,n-2\}$ be such that $J\cup\{n\}$ is $\ell$-short, and $J\cup\{n-1,n\}$ is $\ell$-long. Then with $\bar{J}=\{1,\ldots,n-2\}-J$ we get
\[
 \ell_n-\ell_{n-1}\,\,\,<\,\,\, \ell_{\bar{J}}-\ell_J\,\,\,<\,\,\, \ell_n+\ell_{n-1}
\]
and there exists a unique $x\in S^1=\{(x_1,x_2,0,\ldots,0)\in S^{d-1}\}$ with $x_1>0$, $x_2<0$ such that $(z_1,\ldots,z_{n-1})\in \mathcal{C}_d(\ell)$ with $z_j=x$ for all $j\in J$ and $z_j=-x$ for $j\in \bar{J}$, compare Figure \ref{critconfig}.
\begin{figure}[ht]
\begin{center}
\includegraphics[height=5cm,width=8cm]{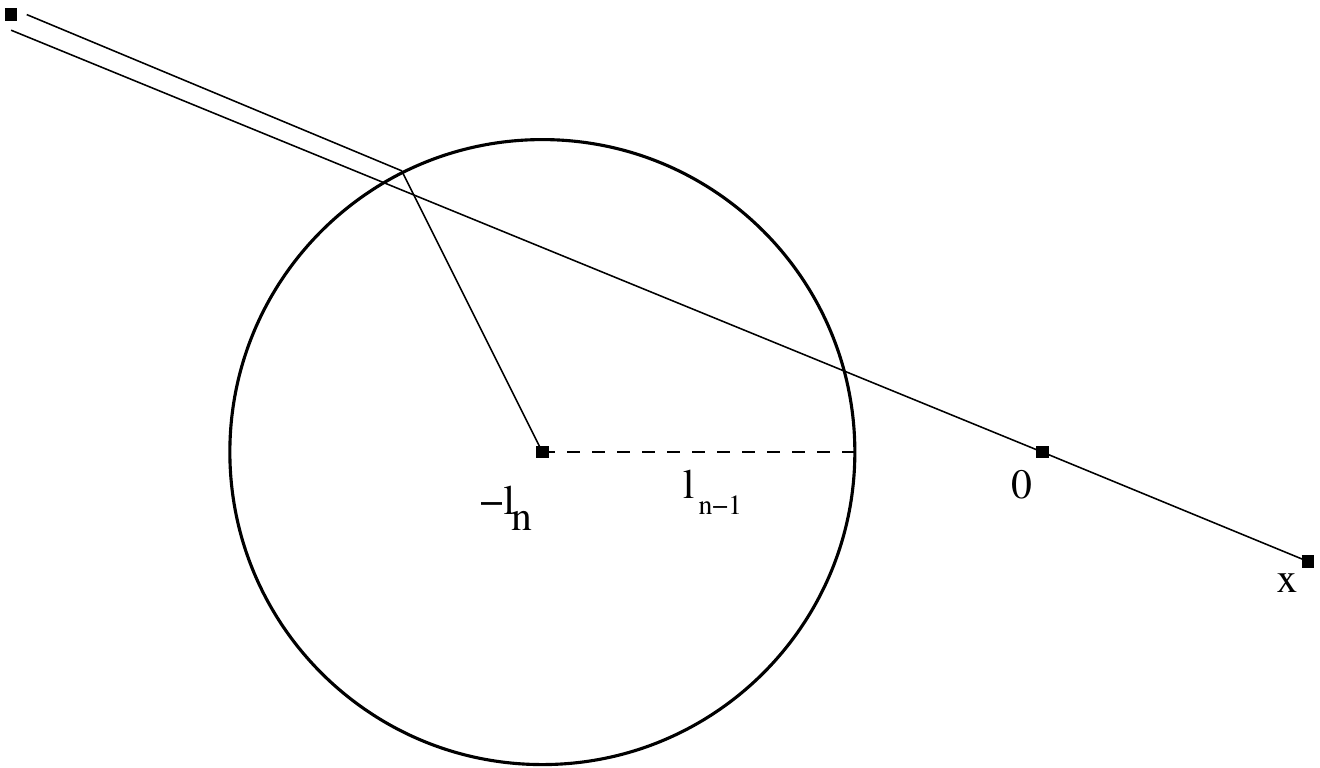}
\caption{\label{critconfig} }
\end{center}
\end{figure}

The orbit under the $SO(d-1)$-action is a sphere of dimension $d-2$ which we denote by 
\begin{eqnarray*}
S_J&\subset &\mathcal{C}_d(\ell).
\end{eqnarray*}

\begin{lemma}
 The critical points of $F$ are given by $\mathcal{C}_d(\ell^\pm)$, and $S_J$ for all $J\subset \{1,\ldots,n-2\}$ for which $J\cup\{n\}$ is $\ell$-short and $J\cup \{n-1,n\}$ is $\ell$-long.
\end{lemma}

\begin{proof}
 We use Lagrange multipliers. Let $f:(\R^d)^{n-1}\times \R^{n-1}\times \R^d \to \R$ be given by
\begin{eqnarray*}
 f(z_1,\ldots,z_{n-1},\lambda,\mu)&=&\ell_{n-1}z_{n-1,1}+\ell_n+\lambda_1(\sum_{j=1}^d z_{1,j}^2-1)+\cdots+\\ 
& &\lambda_{n-1}(\sum_{j=1}^d z_{n-1,j}^2-1)+\mu_1(\sum_{i=1}^{n-1}\ell_i z_{i,1}+\ell_n)+\\
& & \mu_2(\sum_{i=1}^{n-1}\ell_iz_{i,2})+\cdots+ \mu_d(\sum_{i=1}^{n-1}\ell_iz_{i,d})
\end{eqnarray*}
Taking partial derivatives with respect to $z_{k,j}$ and setting them equal to zero leads to equations
\begin{eqnarray*}
 \mu_j&=&-\frac{2\lambda_kz_{k,j}}{\ell_k} \hspace{0.4cm}\mbox{for }(j,k)\not=(1,n-1)\\
 \mu_1&=&-\frac{\ell_{n-1}+2\lambda_{n-1}z_{n-1,1}}{\ell_{n-1}}.
\end{eqnarray*}
For $i=1,\ldots,n-1$, let $\nu_i=\lambda_i/\ell_i$. We have
\begin{eqnarray*}
 z_i\nu_i&=&z_1\nu_1
\end{eqnarray*}
for all $i=1,\ldots,n-2$. Since the $z_i\in S^{d-1}$ we either get that $\nu_i=0$ for all $i=1,\ldots,n-2$ or that $z_1,\ldots,z_{n-2}\in \{\pm x\}$ for some $x\in S^{d-1}$.

The case $\nu_i=0$ for $i=1,\ldots,n-2$ implies $\mu_1=\pm1$ and $\mu_k=0$ for $k\geq 2$, so $z_{n-1}=\pm e_1$, which means that $(z_1,\ldots,z_{n-1})\in \mathcal{C}_d(\ell^\pm)$. These points are clearly critical points of $F$ as they are the maxima and minima.

If the $\nu_i$ are such that $z_1,\ldots,z_{n-2}\in \{\pm x\}$ for some $x\in S^{d-1}$, the condition that $(z_1,\ldots,z_{n-1})\in \mathcal{C}_d(\ell)$ ensures that $(z_1,\ldots,z_{n-1})\in S_J$ for some $J\subset\{1,\ldots,n-2\}$.

Conversely, let $(z_1,\ldots,z_{n-1})\in S_J$. It is straightforward to check that $\mu_j$ and $\lambda_i$ can be chosen so that all partial derivatives of $f$ vanish.
\end{proof}

We want to show that $F$ is Morse-Bott. To do this consider the projection $P:\mathcal{C}_d(\ell)\to S^{d-1}$ given by $P(z)=z_{n-1}$.

\begin{lemma}\label{critofP}
 The critical points of $P:\mathcal{C}_d(\ell)\to S^{d-1}$ are those points for which the first $n-2$ points are collinear.
\end{lemma}

\begin{proof}
 Let $G:(S^{d-1})^{n-2}\to \R^d$ be given by
\begin{eqnarray*}
 G(z_1,\ldots,z_{n-2})&=& \frac{1}{\ell_{n-1}}\left(\ell_ne_1+\sum_{i=1}^{n-2}\ell_iz_i\right)
\end{eqnarray*}
The critical points of $G$ are clearly the collinear points. We have $\mathcal{C}_d(\ell)=G^{-1}(S^{d-1})$, and since $\ell$ is generic, we get that $G$ intersects $S^{d-1}$ transversally. Furthermore, $P$ is just the restriction of $G$ to $\mathcal{C}_d(\ell)$, so if $z\in \mathcal{C}_d(\ell)$ is a regular point for $G$, then $z$ is a regular point for $P$. Also, if $z\in \mathcal{C}_d(\ell)$ is collinear, the rank of $G_\ast$ is $d-1$, and since the intersection with $S^{d-1}$ is transversal, the rank of $P_\ast$ at $z$ is $d-2$. 
\end{proof}

\begin{lemma}\label{maxmincrit}
 For generic $\ell$, the critical submanifolds $\mathcal{C}_d(\ell^\pm)$ are Morse-Bott with respect to $F$. Furthermore, the normal bundle is trivial.
\end{lemma}
 
\begin{proof}
 let $D^{d-1}\subset S^{d-1}$ be a small disc around $\pm e_1$. By Lemma \ref{critofP}, we have $P^{-1}(D^{d-1})\cong D^{d-1}\times \mathcal{C}_d(\ell^\pm)$. The map $F:\mathcal{C}_d(\ell)|P^{-1}(D^{d-1})$ is just a scaling and translation of the standard projection $p_1:S^{d-1}\to \R$ to the first coordinate, restricted to $D^{d-1}$. Since this map is a Morse function with critical points $\pm e_1$, the lemma follows.
\end{proof}

We remark that $\ell$ in the following proposition need not be generic, as the $S_J$ stay away from non-manifold points of $\mathcal{C}_d(\ell)$.

\begin{proposition}\label{indexcalc}
 Let $\ell$ be a length vector and $J\subset \{1,\ldots,n-2\}$ such that $J\cup\{n\}$ is $\ell$-short and $J\cup \{n-1,n\}$ is $\ell$-long. Then $S_J$ is a Morse-Bott critical submanifold of $F$ with index $(n-3-|J|)(d-1)$.
\end{proposition}

\begin{proof}
 Let
\begin{eqnarray*}
 \mathcal{K}_d(\ell)&=&\{(z_1,\ldots,z_{n-1})\in \mathcal{C}_d(\ell)\,|\,z_{n-1}\in S^1\times\{0\}\subset S^{d-1}\}.
\end{eqnarray*}
This has codimension $d-2$ in $\mathcal{C}_d(\ell)$, and $S_J\cap \mathcal{K}_d(\ell)=S^0$ consists of two points. We claim that $f|\mathcal{K}_d(\ell)$ has Morse singularities near $S_J\cap\mathcal{K}_d(\ell)$, and the proposition follows easily from that.

Let $x\in S^1\times\{0\}\subset S^{d-1}$ so that $z_J=(\pm x,\ldots,\pm x,z_{n-1})\in \mathcal{K}_d(\ell)\cap S_J$, where we assume that the sign of $\pm x$ is positive of the coordinate is in $J$, and negative otherwise. Write $x=(\cos \varphi,\sin \varphi)$ and assume $\varphi\in (-\pi/2,0)$, so that Figure \ref{critconfig} applies.

Note that we can write $F|\mathcal{K}_d(\ell)$ as a composition $\mathcal{K}_d(\ell)\stackrel{\tilde{F}}{\longrightarrow}S^1\stackrel{p_1}{\longrightarrow}\R$ with $\tilde{F}$ the projection to $S^1\times \{0\}$. If we replace $p_1:S^1\to \R$ by $p:S^1\to \R$ given by $p(z)=z\cdot x$, that is, scalar product with $x$, it is clear that near $\tilde{F}(z_J)$ we can write $p_1=h\circ p$ where $h$ is an orientation preserving diffeomorphism of open intervals of $\R$.

So to calculate the index of $F$ at $z_J$ we can look at the map $\bar{F}:\mathcal{K}_d(\ell)\to \R$ given by
\begin{eqnarray*}
 \bar{F}(z_1,\ldots,z_{n-1})&=&z_{n-1}\cdot x
\end{eqnarray*}
and calculate its index at the singularity $z_J$.

Note that we have an inclusion $\mathcal{K}_d(\ell)\subset (S^{d-1})^{n-2}$ as those points for which $\sum_{i=1}^{n-2}\ell_iz_i$ sits inside $\R^2\times \{0\}\subset \R^d$ with distance $\ell_{n-1}$ from $(-\ell_n,0,\ldots,0)$. Using the Implicit Function Theorem, we can parametrize $\mathcal{K}_d(\ell)$ near $z_J$ as
\begin{eqnarray*}
 (S^{d-1})^{n-3}&\longrightarrow & (S^{d-1})^{n-2}\\
(u_1,\ldots,u_{n-3}) &\mapsto & (u_1,\ldots,u_{n-3},g(u_1,\ldots,u_{n-3}))
\end{eqnarray*}
where near $z_J$ the $u_i$ are close to $\pm x$.

In this parametrization, the map $\tilde{F}$ is given by
\begin{eqnarray*}
 \tilde{F}(u_1,\ldots,u_{n-3})&=&\frac{-1}{\ell_{n-1}}\left(\ell_ne_1+\sum_{i=1}^{n-3}\ell_iu_i+\ell_{n-2}g(u)\right)\cdot x.
\end{eqnarray*}

Use standard polar coordinates for the $u_i$, that is, we write
\begin{eqnarray*}
 u_i&=&(\sin \theta_{d-1\,i}\cdots \sin \theta_{2\,i}\cos \theta_{1\,i},\sin \theta_{d-1\,i}\cdots \sin \theta_{2\,i}\sin \theta_{1\,i},\\
 & & \sin \theta_{d-1\,i}\cdots \sin \theta_{3\,i}\cos \theta_{2\,i},\ldots,\sin\theta_{d-1\,i}\cos\theta_{d-2\,i},\cos\theta_{d-1\,i})
\end{eqnarray*}
for $i=1,\ldots,n-3$, and $\theta_{1\,i}$ near $\varphi$ or $\varphi+\pi$, depending on whether $i\in J$, and $\theta_{j\,i}$ near $\pi/2$ for $j=2,\ldots,d-1$. The $(n-2)$-nd coordinate can also be written in angles $g_j$ which depend smoothly on the $\theta_{j\,i}$ for all $i=1,\ldots,n-3$ and $j=1,\ldots,d-1$. Let us ignore the factor $(-1)/\ell_{n-1}$ and the translation through $e_1\cdot x$ for now, so that we consider the function
\begin{eqnarray*}
 \tilde{F}&=&\sum_{i=1}^{n-3}\ell_i(\sin \theta_{d-1\,i}\cdots \sin \theta_{2\,i}\cos \theta_{1\,i}\cos\varphi+\sin \theta_{d-1\,i}\cdots \sin \theta_{2\,i}\sin \theta_{1\,i}\sin\varphi)\\
& & +\ell_{n-2} (\sin g_{d-1}\cdots \sin g_2\cos g_1\cos\varphi+\sin g_{d-1}\cdots \sin g_2\sin g_1\sin\varphi)\\
&=&\sum_{i=1}^{n-3}\ell_i\sin \theta_{d-1\,i}\cdots \sin \theta_{2\,i}(\cos(\theta_{1\,i}-\varphi))+\\
& & +\ell_{n-2}\sin g_{d-1}\cdots \sin g_2 (\cos(g_1-\varphi))
\end{eqnarray*}
Writing $\tilde{\theta}_{1\,i}=\theta_{1\,i}-\varphi+\pi/2$ and $\tilde{g}_1=g_1-\varphi+\pi/2$ (and rewriting as $\theta_{1\,i}$ and $g_1$), we get that
\begin{eqnarray*}
 \frac{\partial\tilde{F}}{\partial \theta_{ji}}&=&\ell_i\sin\theta_{d-1\,i}\cdots \cos\theta_{j\,i}\cdots \sin\theta_{2\,i}\sin \theta_{1\,i}\\
& & + \sum_{k=1}^{d-1}\ell_{n-2}\sin g_{d-1}\cdots \cos g_k \cdots \sin g_2 \sin g_1 \frac{\partial g_k}{\partial g_{j\,i}}
\end{eqnarray*}
Note that the point $z_J$ now corresponds to all angles being $\pi/2$ or $3\pi/2$, so that the cosine terms always vanish. At the point $z_J$, we therefore get
\begin{eqnarray*}
 \frac{\partial^2\tilde{F}}{\partial\theta_{j\,i}\partial\theta_{j\,i}}&=& -\ell_i\sin\theta_{1\,i}-\sum_{k=1}^{d-1}\ell_{n-2}\sin g_1\left(\frac{\partial g_k}{\partial\theta_{j\,i}}\right)^2
\end{eqnarray*}
and for $(j,i)\not=(l,m)$ we get
\begin{eqnarray*}
  \frac{\partial^2\tilde{F}}{\partial\theta_{j\,i}\partial\theta_{l\,m}}&=& -\sum_{k=1}^{d-1}\ell_{n-2}\sin g_1\frac{\partial g_k}{\partial\theta_{j\,i}}\frac{\partial g_k}{\partial\theta_{l\,m}}
\end{eqnarray*}
where $\sin \theta_{1\,i}=1$ for $i\in J$, $-1$ for $i\notin J$ and $\sin g_1=1$ for $n-2\in J$ and $-1$ for $n-2\notin J$. If we write $\delta_i=+1$ for $i\in J$ and $\delta_i=-1$ for $i\notin J$ ($i\leq n-2$), it follows from Lemma \ref{gcalc} below that
\begin{eqnarray*}
 \frac{\partial^2\tilde{F}}{\partial\theta_{j\,i}\partial\theta_{j\,i}}&=&  -\ell_i\delta_i-\delta_{n-2}\frac{\ell_i^2}{\ell_{n-2}}\\
\frac{\partial^2\tilde{F}}{\partial\theta_{j\,i}\partial\theta_{l\,m}}&=& 0 \hspace{0.4cm}\mbox{for }j\not=l\\
\frac{\partial^2\tilde{F}}{\partial\theta_{j\,i}\partial\theta_{j\,m}}&=& \left\{
\begin{array}{rl}
-\delta_{n-2}\delta_i\delta_m\frac{\ell_i\ell_m}{\ell_{n-2}}&j=1\\
-\delta_{n-2}\frac{\ell_i\ell_m}{\ell_{n-2}}& j\geq 2
\end{array}
\right.
\end{eqnarray*}
The matrix $\left(\frac{\partial^2\tilde{F}}{\partial\theta_{j\,i}\partial\theta_{l\,m}}\right)$ is a $(n-3)(d-1)\times (n-3)(d-1)$ matrix, which we consider as a $(d-1)\times (d-1)$ matrix with entries $(n-3)\times (n-3)$ matrices $\left(\frac{\partial^2\tilde{F}}{\partial\theta_{j\,i}\partial\theta_{l\,m}}\right)_{i,m}$ for fixed $j,l$. The off-diagonal entries are then $0$, while the diagonal entries are matrices $\left(\frac{\partial^2\tilde{F}}{\partial\theta_{j\,i}\partial\theta_{j\,m}}\right)_{i,m}$ for $j=1,\ldots,d-1$. These matrices are of the form
\begin{eqnarray*}
 \left(\frac{\partial^2\tilde{F}}{\partial\theta_{1\,i}\partial\theta_{1\,m}}\right)_{i,m}&=& -\Delta(\ell_1\delta_1,\ldots,\ell_{n-3}\delta_{n-3})-\frac{\delta_{n-2}}{\ell_{n-2}}\left(\delta_i\ell_i\delta_m\ell_m\right)_{i,m}\\
\left(\frac{\partial^2\tilde{F}}{\partial\theta_{j\,i}\partial\theta_{j\,m}}\right)_{i,m}&=& -\Delta(\ell_1\delta_1,\ldots,\ell_{n-3}\delta_{n-3})-\frac{\delta_{n-2}}{\ell_{n-2}}\left(\ell_i\ell_m\right)_{i,m}
\end{eqnarray*}
for $j=2,\ldots,d-1$. Here $\Delta$ is a diagonal matrix with the given entries.

Since $|J|\leq n-3$, we can assume (possibly after rearranging the order of the links) that $n-2\notin J$, that is, $\delta_{n-2}=-1$. It follows that these matrices are congruent to
\begin{eqnarray*}
 M&=& -\frac{1}{\ell_{n-2}}\left( \Delta(\frac{\ell_{n-2}\delta_1}{\ell_1},\ldots,\frac{\ell_{n-2}\delta_{n-3}}{\ell_{n-3}}) - E\right)
\end{eqnarray*}
where $E$ has every entry equal to $1$. Recall that we ignored a factor $(-1)/\ell_{n-1})$ in $\tilde{F}$ above, so we need to calculate the index of $\Delta(\frac{\ell_{n-2}\delta_1}{\ell_1},\ldots,\frac{\ell_{n-2}\delta_{n-3}}{\ell_{n-3}}) - E$. A calculation as in \cite[Lemma 1.4]{farbin} shows that the index is $n-3-|J|$. Since we have $d-1$ such matrices, the result follows.
\end{proof}

\begin{lemma}\label{gcalc}
With notation as in Proposition \ref{indexcalc}, we have
 \begin{eqnarray*}
  \frac{\partial g_k}{\partial\theta_{j\,i}}&=& 0 \hspace{0.4cm}\mbox{for }k\not=j\\
  \frac{\partial g_k}{\partial\theta_{k\,i}}&=& \left\{\begin{array}{rl}
                                                      \delta_i\frac{\ell_i}{\ell_{n-2}} & k=1\\
                                                      \frac{\ell_i}{\ell_{n-2}} & k\geq 2
                                                     \end{array}\right.
 \end{eqnarray*}
\end{lemma}

\begin{proof}
 Define $G:(S^{d-1})^{n-2}\to \R^{d-1}$
\begin{eqnarray*}
 G(u_1,\ldots,u_{n-2})&=&\left(\begin{array}{c}
                                \left|\sum_{i=1}^{n-2}\ell_iu_i+\ell_ne_1\right|^2\\
                                p_3\left(\sum_{i=1}^{n-2}\ell_iu_i\right)\\
                                \vdots \\
                                p_d\left(\sum_{i=1}^{n-2}\ell_iu_i\right)
                               \end{array}
\right)
\end{eqnarray*}
Then $\mathcal{K}_d(\ell)=G^{-1}(\ell_{n-1},0,\ldots,0)$. In polar coordinates a calculation shows that
\begin{eqnarray*}
 G_1&=&\ell_n^2+\sum_{i=1}^{n-2}\ell_i^2+2\ell_n \sum_{i=1}^{n-2}\ell_i\sin \theta_{d-1\,i}\cdots \sin\theta_{2\,i} \cos \theta_{1\,i}\\
& &+2\sum_{i<j}\ell_i\ell_j (\sin\theta_{d-1\,i}\cdots \sin\theta_{2\,i}\sin_{d-1\,j}\cdots \sin\theta_{2\,j}(\cos(\theta_{1\,i}-\theta_{1\,j})\\
& & \hspace{2cm} +\sin\theta_{d-1\,i}\cdots \sin\theta_{3\,i}\cos\theta_{2\,i}\sin\theta_{d-1\,j}\cdots \sin\theta_{3\,j}\cos\theta_{2\,j}\\
& & \hspace{2cm} +\cdots + \cos\theta_{d-1\,i}\cos \theta_{d-1\,j}).
\end{eqnarray*}
Similarly, for $k\geq 2$ we have
\begin{eqnarray*}
 G_k&=&\sum_{i=1}^{n-2}\ell_i\sin\theta_{d-1\,i}\cdots \sin\theta_{k+1\,i}\cos\theta_{k\,i}.
\end{eqnarray*}
Using the fact that $z_J$ has $\theta_{j\,i}=\pi/2$ for $j>1$, it is easy to see that
\begin{eqnarray*}
 \frac{\partial G_1}{\partial\theta_{1\,i}}(z_J)&=&\left\{\begin{array}{rl}
                                                           2\ell_n\ell_i \sin\varphi & i\in J\\
                                                           -2\ell_n\ell_i\sin \varphi& i\notin J
                                                          \end{array}
\right.\\
\frac{\partial G_k}{\partial\theta_{j\,i}}(z_J)&=& 0 \hspace{0.4cm}\mbox{for }j\not=k\\
\frac{\partial G_k}{\partial\theta_{k\,i}}(z_J)&=& -\ell_i \hspace{0.4cm}\mbox{for }k\geq 2
\end{eqnarray*}
So for fixed $i\leq n-2$, each $\left(\frac{\partial G_k}{\partial\theta_{j\,i}}(z_J)\right)$ is an invertible diagonal matrix. In particular, by the Implicit Function Theorem we get for $i\leq n-3$
 \begin{eqnarray*}
  \frac{\partial g_k}{\partial\theta_{j\,i}}&=& 0 \hspace{0.4cm}\mbox{for }k\not=j\\
  \frac{\partial g_k}{\partial\theta_{k\,i}}&=& \left\{\begin{array}{rl}
                                                      \delta_i\frac{\ell_i}{\ell_{n-2}} & k=1\\
                                                      \frac{\ell_i}{\ell_{n-2}} & k\geq 2
                                                     \end{array}\right.
 \end{eqnarray*}
since the $g_k$ are obtained by applying the Implicit Function Theorem to $G$.
\end{proof}

\section{Homology for the 3-dimensional case}
\label{perfectthree}

In this section we show how the Betti numbers for $\mathcal{M}_3(\ell)$ can be obtained from the Morse-Bott function above. We will only sketch the argument, as these results have already been obtained in \cite{klyach}. Information on the cohomology is contained in \cite{hauknu}.

For $d=3$ and generic $\ell$, the action of $SO(2)$ is free on $\mathcal{C}_3(\ell)$, and $\mathcal{M}_3(\ell)$ is a closed manifold. Furthermore, the $SO(2)$-invariant function $F$ induces a Morse-Bott function $f:\mathcal{M}_3(\ell)\to \R$, which has $\mathcal{M}_3(\ell^-)$ as minimum, $\mathcal{M}_3(\ell^+)$ as maximum (with index $2$), and for each $J\subset\{1,\ldots,n-2\}$ with $J\cup\{n\}$ short and $J\cup\{n-1,n\}$ long a critical point $p_J$ of index $2(n-3-|J|)$.

A simple induction argument using the Morse-Bott spectral sequence shows that the homology of $\mathcal{M}_3(\ell)$ is free abelian and concentrated in even degrees. If we write $P_\ell(t)$ for the Poincar\'e polynomial of $\mathcal{M}_3(\ell)$, we get the following recursive formula.

\begin{proposition}
 Let $\ell\in \R^n$ be a generic length vector. Then the Poincar\'e polynomial of $\mathcal{M}_3(\ell)$ satisfies
\begin{eqnarray*}
 P_\ell(t)&=&P_{\ell^-}(t)+t^2P_{\ell^+}(t)+\sum_{J\subset\mathcal{T}(\ell)}t^{2|J|}
\end{eqnarray*}
where $\mathcal{T}(\ell)=\{J\subset\{1,\ldots,n-2\}\,|\,J\cup \{n\} \mbox{ short, }J\cup\{n-1,n\} \mbox{ long}\}$.
\end{proposition}

\begin{remark}
 A similar recursive formula is obtained in \cite[Cor.2.2.2]{klyach} by using different methods. In fact, we can get that formula by looking at $-f$ instead of $f$. Klyachko goes on to give the following explicit formula for the Poincar\'e polynomial, see \cite[Thm.2.2.4]{klyach}.
\begin{eqnarray*}
 P_\ell(t)&=&\frac{1}{t^2(t^2-1)}\left((1+t^2)^{n-1}-\sum_{J\in\mathcal{S}(\ell)}t^{2|J|}\right),
\end{eqnarray*}
where $\mathcal{S}(\ell)=\{J\subset\{1,\ldots,n\}\,|\,J\mbox{ short }\}$.

Hausmann and Knutson \cite[Cor.4.3]{hauknu} derive another formula, given by
\begin{eqnarray*}
 P_\ell(t)&=&\frac{1}{1-t^2}\sum_{J\in\mathcal{S}_\ast^n(\ell)}(t^{2|J|}-t^{2(n-2-|J|)}),
\end{eqnarray*}
where we assume that $\ell$ is $n$-dominated.
\end{remark}

\begin{corollary}\label{perfect3}
 Let $\ell\in \R^n$ be a generic length vector. Then there exists a perfect Morse function $f_3:\mathcal{M}_3(\ell)\to \R$, all of whose critical points are of even index.
\end{corollary}

\begin{proof}
 The proof is by induction on $n$, using standard techniques for replacing the Morse-Bott manifolds $\mathcal{M}_3(\ell^-)$ and $\mathcal{M}_3(\ell^+)$ by (perfect) Morse functions. This gives a Morse function with all indices of critical points even.
\end{proof}

Let us give a formula for the number of critical points of a given index. For this let $\mu_k(\ell)$ be the number of critical points of of $f_3$ having index $2k$.

\begin{proposition}
 \label{threeBetti}
Let $\ell\in \R^n$ be a generic length vector, and let $m\in \Z$ be such that $n=2m-1$ or $n=2m$. Then
\begin{eqnarray*}
 \mu_k(\ell)&=& \sum_{i=0}^k a_i(\ell) - a_{n-2-i}(\ell)
\end{eqnarray*}
for all $k=0,\ldots, m-2$, and
\begin{eqnarray*}
 \mu_k(\ell)&=& \mu_{n-3-k}(\ell)
\end{eqnarray*}
for all $k=m-1,\ldots,n-3$.
\end{proposition}

\begin{proof}
The second equation just follows from Poincar'e duality, and the first equation is a straightforward application of the formula of \cite{hauknu}.
\end{proof}

In particular, we have
\begin{eqnarray}\label{poincare3}
 P_\ell(t)&=& \sum_{i=0}^{m-2}(a_i(\ell)-a_{n-2-i}(\ell))(t^{2i}+t^{2(i+1)}+\cdots+t^{2(n-3-i)}).
\end{eqnarray}

\section{An equivariant cell decomposition for $\mathcal{C}_d(\ell)$}
\label{sec6}

We want to derive an equivariant cell decomposition for $\mathcal{C}_d(\ell)$ using the Morse-Bott function $F$ in order to get a cell decomposition for $\mathcal{M}_d(\ell)$ for $d\geq 3$.

To do this we first want to understand the equivariant handle structure near a critical manifold $S_J$ in the sense of \cite{wasser}. If $J\subset\{1,\ldots,n-2\}$ has the property that $S_J$ is a critical sphere, let $x\in S^1\times\{0\}$ be such that $p_J=(\pm x,\ldots,\pm x,z_{n-1})\in S_J$ and the minus signs correspond to coordinates from $J$. We may assume that $n-2\in J$. Let $D^{d-1}\subset S^{d-1}$ be a small disc with center at $-x$, and define
\begin{eqnarray*}
 P:(D^{d-1})^{n-3-|J|} & \longrightarrow & \mathcal{K}_d(\ell)\\
(u_1,\ldots,u_{n-3-|J|})& \mapsto & (v_1,\ldots,v_{n-3},g(u_1,\ldots,u_{n-3-|J|}))
\end{eqnarray*}
where $v_i=x$ if $i\notin J$, $v_i=u_{k_i}$ for $i\in J=\{k_1,\ldots,k_{n-3-|J|}\}$. That is, we use the parametrization of $\mathcal{K}_d(\ell)$ from the proof of Proposition \ref{indexcalc}, but we keep the coordinates away from $J$ fixed.

By the same argument as in the proof of Proposition \ref{indexcalc}, $F\circ P$ has a nondegenerate maximal point at $(-x,\ldots,-x)$, which is the center of $(D^{d-1})^{n-3-|J|}$.

For simplicity, let us center $D^{d-1}$ at $0$, and we think of $P$ as an inclusion $i:(D^{d-1})^{n-3-|J|}\to \mathcal{C}_d(\ell)$. If we let $SO(d-2)$ act diagonally on $(D^{d-1})^{n-3-|J|}$, with $SO(d-2)$ acting in a standard way on $D^{d-1}\subset\R^{d-1}$ by fixing the first coordinate, we get that $i$ is $SO(d-2)$-equivariant.

The image of $i$ is in $\mathcal{K}_d(\ell)$, and by using the action of $SO(d-1)$ on the image, we get the negative normal bundle of $S_J$ in the sense of equivariant Morse theory, compare \cite{wasser}. We thus write
\begin{eqnarray*}
 N^-(S_J)&=&\left\{A i(x)\in \mathcal{C}_d(\ell)\,|\,A\in SO(d-1), x\in (D^{d-1})^{n-3-|J|}\right\}.
\end{eqnarray*}
The map $N^-(S_J)\to SO(d-1)/SO(d-2)\cong S_J$ given by $Ai(x)\mapsto A\cdot SO(d-2)$ is then a disc bundle map with fibre $(D^{d-1})^{n-3-|J|}$.

We want to have an equivariant Morse-Bott function $\tilde{F}:\mathcal{C}_d(\ell)\to \R$ such that all critical manifolds are spheres $SO(d-1)/SO(d-2)$ with negative normal bundle as the $N^-(S_J)$. The idea is to use the argument in the proof of Corollary \ref{perfect3}, but equivariantly. This can be done, as there are neighborhoods of $\mathcal{C}_d(\ell^\pm)$ in $\mathcal{C}_d(\ell)$ which are equivariantly diffeomorphic to $\mathcal{C}_d(\ell)\times D^{d-1}$, compare Lemma \ref{maxmincrit}. We use the fact that for $\ell$ and $\ell'$ in the same chamber the chain spaces are equivariantly diffeomorphic \cite{hausma}. Notice that the critical manifolds do not depend on $d$. We thus get the following result:

\begin{proposition}\label{perfectmorse}
 Let $\ell\in \R^n$ be a generic length vector. For all $d\geq 3$ there is an $SO(d-1)$ invariant Morse-Bott function $\tilde{F}:\mathcal{C}_d(\ell)\to \R$ such that all critical manifolds are of the form $SO(d-1)/SO(d-2)$, and their indices are of the form $k(d-1)$ for some $k=0,\ldots,n-3$. The negative normal bundle $N^-(S)$ to each critical manifold $S$ is of the form $(D^{d-1})^k \longrightarrow N^-(S)\longrightarrow SO(d-1)/SO(d-2)$, where $SO(d-2)$ acts on $(D^{d-1})^k$ diagonally, fixing the first coordinate of $D^{d-1}$.

Furthermore, the critical manifolds $S$ of index $k(d-1)$ are in one-to-one correspondence to the critical points of index $2k$ of the perfect Morse function $f_3:\mathcal{M}_3(\ell)\to \R$ from Corollary \ref{perfect3}.\hfill \qed
\end{proposition}

Denote by $\partial N^-(S)$ the sphere bundle corresponding to $N^-(S)$. In order to understand the homotopy type of $\mathcal{M}_d(\ell)$ we want to understand a relative $SO(d-1)$-equivariant cell structure on $(N^-(S),\partial N^-(S))$. Since $N^-(S)$ is the $SO(d-1)$ orbit of the image of $(D^{d-1})^k$, we have to find a relative $SO(d-2)$-equivariant cell structure of $((D^{d-1})^k,\partial (D^{d-1})^k)$.

Let us begin with some elementary observations. If $k=1$, the set $D^1\times \{0\}\subset D^{d-1}$ is the fixed set of the $SO(d-2)$-action. It therefore defines a $1$-cell with $SO(d-2)$ as the stabilizer group. If $x\in D^{d-1}-D^1\times\{0\}$, we can find an $A\in SO(d-2)$ such that $Ax=(a,b,0,\ldots,0)\in D^{d-1}$, with $b\not=0$. If $d>3$, we can furthermore assume that $b>0$. In particular, every other element of $D^{d-1}$ will be in the orbit of an element of $D^2_+=\{(a,b,0,\ldots,0)\in D^{d-1}\,|\, b\geq 0, a^2+b^2\leq 1\}$.

In particular, we only need two cells. If we denote $X=D^{d-1}/SO(d-2)$ and $\partial X=\partial D^{d-1}/SO(d-2)$, we get a relative CW-structure of $(X,\partial X)$ with $X$ being obtained from $\partial X$ by an elementary expansion in the sense of \cite[\S 4]{mcohen}. If $d=3$, note that $SO(d-2)$ is the trivial group. We either have to use two $2$-cells (one for $b>0$ and one for $b<0$), or we do not use the $1$-cell, and just use the $2$-cell $D^2$.

We can ignore the case $d=3$, in which we only need one cell for $((D^2)^k,\partial (D^2)^k)$ of dimension $2k$. So assume $d\geq 4$ now. Let $(x_1,\ldots,x_k)\in (D^{d-1})^k$. After applying an element of $SO(d-2)$ we can assume $x_1\in D^2_+$. If we actually have $x_1\in D^1$, we apply another element of $SO(d-2)$ to get $x_2\in D^2_+$. We repeat this until we get an element $x_i\in D^2_+-D^1$. If we do not get such an element, the original element $(x_1,\ldots,x_k)$ is in $(D^1)^k$. So assume $x_i\in D^2_+-D^1$ and $x_j\in D^1$ for $j<i$. Applying an element of $SO(d-3)$ does not affect the first $i$ elements, and can move $x_{i+1}$ into $D^3$, in fact $D^3_+$ if $d>4$. We can continue this so we may assume that up to elements of $SO(d-2)$, the element $(x_1,\ldots,x_k)$ is in a product of an increasing sequence of discs.

To make this more precise, write
\begin{eqnarray*}
 D^i&=&\{(x_1,\ldots,x_i,0,\ldots,0)\in D^{d-1}\,|\,x_1^2+\ldots,x_i^2\leq 1\}
\end{eqnarray*}
for $i=1,\ldots, d-2$, and also write
\begin{eqnarray*}
 D^i_+&=& \{(x_1,\ldots,x_i,0,\ldots,0)\in D^{d-1}\,|\,x_1^2+\ldots,x_i^2\leq 1, x_i\geq 0\}
\end{eqnarray*}
for $i=2,\ldots,d-2$.

Up to an element of $SO(d-2)$, any $(x_1,\ldots,x_k)\in (D^{d-1})^k$ sits in
\[
 (D^1)^{k_1} \times (D^2)^{k_2} \times \cdots \times (D^{d-2})^{k_{d-2}} \times (D^{d-1})^{k_{d-1}}
\]
where all $k_i\geq 0$ and add up to $k$. Furthermore, if $k_i=0$ for $i\geq 2$, then all $k_j=0$ with $j\geq i$, and if $k_i\not=0$ for $i\in\{2,\ldots,d-2\}$, we can replace $(D^i)^{k_i}$ by $D^i_+\times (D^i)^{k_i-1}$.

In order to organise the cells we introduce symbolic matrices. For $n,m\geq 1$ let $\mathbf{S}(m,n)$ be the set of upper semi-diagonal $n\times m$ matrices whose entries are from the set $\{0,+,\ast\}$, which have a $+$ sign for the first non-zero entry in each of the first $n-1$, with all entries to the right of the $+$ as $\ast$, and the last row contains only $0$ and $\ast$, with no $0$ to the right of any $\ast$.

Typical examples are
\begin{equation}\label{symbmatrices}
 \left(\begin{array}{ccccccc}
        0 & 0 & + & \ast & \ast & \ast & \ast \\
        0 & 0 & 0 & + & \ast & \ast & \ast \\
        0 & 0 & 0 & 0 & 0 & \ast & \ast
       \end{array}
 \right)
\, , \,
\left(\begin{array}{cccc}
       + & \ast & \ast & \ast \\
       0 & 0 & + & \ast \\
       0 & 0 & 0 & 0 
      \end{array}
\right).
\end{equation}

Each matrix stands for a product of discs, with columns refering to each disc. Here the zero column stands for $D^1$, a column containing a $+$ stands for $D^k_+$ and a column with only $\ast$ and $0$ stands for $D^k$, where $k-1$ is the number of non-zero entries in the column.

\begin{remark}
Such symbolic matrices were already used in \cite{kebacl} to get a cell decomposition of the shape spaces $\Sigma^m_d$, and our homology calculations below are indeed quite similar to the calculations in \cite{kebacl}.
\end{remark}

So if $A\in \mathbf{S}(d-2,k)$, we denote by $D_A\subset (D^{d-1})^k$ the corresponding product of discs. Also, let $SO(A)\subset SO(d-2)$ be the stabilizer group of $D_A$. Then $SO(A)=SO(d-2-i)$, where $i$ is the maximal number of non-zero elements in the columns of $A$, and $SO(d-2-i)$ acts on $D^{d-1}$ by fixing the first $i+1$ coordinates. In particular $SO(0)$ is the trivial group. We denote the image of $D_A$ under the $SO(d-2)$-action by $\sigma_A$, and call this the cell corresponding to $A$.

\begin{lemma}
 Every interior point $x\in (D^{d-1})^k$ is contained in the interior of a cell corresponding to a symbolic matrix $A\in \mathbf{S}(d-2,k)$. If the stabilizer of $x$ is non-trivial, this cell is unique.
\end{lemma}

\begin{proof}
 The proof is by induction on $k$. If $k=1$, there is only the zero matrix and a matrix with one non-zero entry. It is easy to see that the result holds in this case.

Now let $x=(x_1,\ldots,x_k)\in (D^{d-1})^k$ with $k>1$. If $x_1\in D^1$, we can use induction on $x'=(x_2,\ldots,x_k)$ to get a matrix $A'\in \mathbf{S}(d-2,k-1)$ so that $x'\in \sigma_{A'}^o$. $A'$ is unique if the stabilizer of $x'$ is non-trivial. Then $x\in \sigma_A$, where $A$ is the matrix obtained from $A'$ by adding a zero column to the left of $A'$. Note that the stabilizer of $x$ is the stabilizer of $x'$, and the uniqueness applies if it is non-trivial.

If $x_1\notin D^1$, we can find an $A\in SO(d-2)$ such that $Ax_1\in D^2_+$. Now let $p:D^{d-1}\to D^{d-2}$ be projection to the last $d-2$ coordinates and consider the point $x'=(p(Ax_2),\ldots,p(Ax_k))$. By induction, we can find a symbolic matrix $A'\in \mathbf{S}(d-3,k-1)$ with $x'\in \sigma_{A'}$, and the matrix is unique if the stabilizer of $x'$ is non-trivial (which implies that the stabilizer of $x$ is non-trivial). Then $x$ is in the cell $\sigma_A$, where
\begin{eqnarray*}
 A&=&\left( \begin{array}{cc}
                  + & \ast \\
                  0 & A'
                 \end{array}
\right),
\end{eqnarray*}
and the cell is unique if the stabilizer of $x$ is non-trivial.
\end{proof}

\begin{lemma}\label{fullcells}
 Let $x\in (D^{d-1})^k$ be an interior point with trivial stabilizer. Then $x$ is contained in a cell $\sigma_A$ where the last two rows of $A$ are of the form
\begin{equation}\label{fulcel}
 \left( \begin{array}{ccccccc}
         0 & \cdots & 0 & + & \ast & \cdots & \ast \\
         0 & \cdots & 0 & 0 & \ast & \cdots & \ast
        \end{array}
\right).
\end{equation}
Furthermore, no two such matrices have interior points in common.
\end{lemma}

\begin{proof}
 We know from the previous lemma that $x$ is contained in some matrix, and since the stabilizer of $x$ is trivial, the second but last row has to be non-zero. In particular, there has to be a $+$ in that row. Since $\ast$ symbolizes any possible entry, including $0$, $x$ will be in a cell corresponding to such a matrix.

To see that no two such matrices have interior points in common, note that in the column which has a $+$ in second but last row, interior points $y\in D^{d-2}_+$ satisfy $y_{d-2}>0$, and this is the first column, for which this occurs.
\end{proof}

Notice that the matrices in (\ref{symbmatrices}) are not in the form of Lemma \ref{fullcells}.

Define
\begin{eqnarray*}
 \mathbf{S}_c(d-2,k) & = & \{ A\in \mathbf{S}(d-2,k)\,|\, \mbox{The last two rows are of the form (\ref{fulcel}) or }0\}
\end{eqnarray*}

An equivariant relative cell decomposition of $((D^{d-1})^k,\partial (D^{d-1})^k)$ is therefore given by the cells $\sigma_A$, where $A\in \mathbf{S}_c(d-2,k)$.

\section{The boundary operator for the cell decomposition}
\label{boundarysection}

The equivariant cell decomposition described in the previous section gives a relative CW-structure on $(X_d^k,\partial X_d^k)$, where $X_d^k=(D^{d-1})^k/SO(d-2)$ and $\partial X_d^k = \partial (D^{d-1})^k/SO(d-2)$. The cells are simply of the form $D_A$ for $A\in \mathbf{S}_c(d-2,k)$, where each $D_A$ is a product of discs or halfdiscs.

Notice that the boundary of each factor $D^i$ is attached to $\partial X_d^k$, and each factor $D^i_+$ is attached to $\partial X_d^k$ and to the same cell with the factor replaced by $D^{i-1}$. So the boundary of a cell $D_A$ is contained in $\partial X_d^k$ together with cells $D_{A'}$, where the $A'$ are obtained from $A$ by replacing a $+$ by a $0$.

This needs to be made slightly more precise. If a matrix $A$ contains a submatrix $\left(\begin{array}{cc} + & \ast \\ 0 & + \end{array} \right)$, replacing the $+$ in the upper left corner by $0$ leads to a matrix with submatrix $\left(\begin{array}{cc} 0 & \ast \\ 0 & + \end{array} \right)$, which is not an element of $\mathbf{S}_c$. However, up to elements of $SO(d-2)$ we get that the corresponding boundary points are in the cell containing the submatrix $\left(\begin{array}{cc} 0 & + \\ 0 & 0 \end{array} \right)$. The dimension of this cell is the dimension of the original cell $-2$. In particular, it will not occur in the boundary operator.

If the matrix $A$ contains a submatrix $\left(\begin{array}{cc} + & \ast \\ 0 & 0 \end{array} \right)$, replacing the $+$ in the upper left corner by $0$ leads to the matrix with submatrix $\left(\begin{array}{cc} 0 & + \\ 0 & 0  \end{array} \right)$, but the change from $\ast$ to $+$ in the right upper corner means that the attaching is done twice, so the coefficient in the boundary operator is $0$ or $2$, depending on orientation considerations.

Finally, if the last non-zero row of the matrix $A$ is of the form $(\begin{array}{cccc}0 & \cdots & 0 & + \end{array})$, replacing this row by the zero row gives a matrix $A'\in \mathbf{S}_c$, and the corresponding coefficient in the boundary operator is $\pm 1$.

For $i=1,\ldots,d-3$, define $\mathbf{S}^{(i)}(d-2,k)$ to consist of those matrices $A\in \mathbf{S}_c(d-2,k)$ for which the $i$-th row is $(\begin{array}{ccc} 0 & \cdots & 0 \end{array})$ or $(\begin{array}{cccc} 0 & \cdots & 0 & + \end{array})$.

Let $(X^k_{(i),d},\partial X_d^k)$ be the relative CW-complex consisting of the cells corresponding to $\mathbf{S}^{(i)}(d-2,k)$.

\begin{lemma}\label{collapse}
 The relative CW-complex $(X^k_{(i),d},\partial X_d^k)$ collapses to $(\partial X_d^k,\partial X_d^k)$ for all $i=1,\ldots,d-3$.
\end{lemma}

\begin{proof}
 The proof is by induction on $i$. For $i=1$, we only have two cells, corresponding to the zero matrix and the matrix whose only non-zero entry is a $+$. By the discussion above, the two cells form an elementary collapse in the sense of \cite[\S 4]{mcohen}, and the result follows.

For $i>1$ we show that $(X^k_{(i),d},\partial X_d^k)$ collapses to $(X^k_{(i-1),d},\partial X_d^k)$. Note that if $i>k$, then $X^k_{(i),d}=X^k_{(i-1),d}$ and there is nothing to show. So assume $i\leq k$ and let $A\in \mathbf{S}^{(i)}(d-2,k)-\mathbf{S}^{(i-1)}(d-2,k)$. Then the $(i-1)$-th row of $A$ is non-zero, and different from $(\begin{array}{cccc} 0 & \cdots & 0 & + \end{array})$. The $i$-th row is either $(\begin{array}{ccc} 0 & \cdots & 0 \end{array})$ or $(\begin{array}{cccc} 0 & \cdots & 0 & + \end{array})$, and the two possibilities form an elementary collapse. By collapsing these pairs in the order of decreasing dimension, we see that $(X^k_{(i),d},\partial X_d^k)$ collapses to $(X^k_{(i-1),d},\partial X_d^k)$. The result follows.
\end{proof}

\begin{corollary}
 For $k<d-2$, the pair $(X_d^k,\partial X_d^k)$ is $m$-connected for all $m\geq 0$.
\end{corollary}

\begin{proof}
 The cells of the relative CW-complex are in one-to-one correspondence with $\mathbf{S}_c(d-2,k)$, but since $k<d-2$, we get $\mathbf{S}_c(d-2,k)=\mathbf{S}^{(d-3)}(d-2,k)$. The result thus follows from Lemma \ref{collapse}.
\end{proof}

In the next result the condition $a_0(\ell)=1$ is needed to avoid the case $\mathcal{M}_d(\ell)=\emptyset$.

\begin{proposition}\label{contractible}
 Let $\ell\in \R^n$ be a generic length vector with $a_0(\ell)=1$, and $d\geq n\geq 3$. Then $\mathcal{M}_d(\ell)$ is contractible.
\end{proposition}

\begin{proof}
 Let $\tilde{F}:\mathcal{C}_d(\ell)\to \R$ be the $SO(d-1)$-invariant Morse-Bott function of Proposition \ref{perfectmorse}, $F:\mathcal{M}_d(\ell)\to \R$ the induced function and let
\[
 \emptyset = \mathcal{M}^0 \subset \mathcal{M}^1 \subset \cdots \subset \mathcal{M}^m = \mathcal{M}_d(\ell)
\]
be a filtration such that $\mathcal{M}^m= F^{-1}((-\infty,a_m])$ for some sequence of regular values of $\tilde{F}$ such that $\mathcal{M}^m-\mathcal{M}^{m-1}$ contains exactly one critical point.

By Morse-Bott theory, $\mathcal{M}^m$ is homotopy equivalent to $\mathcal{M}^{m-1}\cup X_d^k$, where $X_d^k$ is attached to $\mathcal{M}^{m-1}$ along $\partial X_d^k$, and $k$ is such that $k(d-1)$ is the index of the critical point in $\mathcal{M}^{m}-\mathcal{M}^{m-1}$. Since $k\leq n-3$, we get $k<d-2$, and $\mathcal{M}^m$ has the same homotopy type as $\mathcal{M}^{m-1}$, provided $k\geq 1$. As there is a unique minimum for $F$ by the perfectness of the map $F_3$ in Proposition \ref{perfectmorse}, we get that $\mathcal{M}^1$ has the homotopy type of a point, and all other critical points have index bigger than $0$.
\end{proof}

Of course, by \cite{schoen} these spaces are homeomorphic to a disc.

If $k\geq d-2$, then $\mathbf{S}^{(d-3)}(d-2,k)\not= \mathbf{S}_c(d-2,k)$. A matrix $A\in \mathbf{S}_c(d-2,k)-\mathbf{S}^{(d-3)}(d-2,k)$ has to have at least one $\ast$ in its last row, and therefore it has $(d-1)(d-2)/2$ non-zero entries. It follows that $D_A$ has at least dimension $(d-1)(d-2)/2+k$.

\begin{lemma}\label{connectedness}
 Let $k\geq d-2\geq 2$, then $(X_d^k,\partial X_d^k)$ is $((d-1)(d-2)/2 + k-1)$-connected, but not $((d-1)(d-2)/2 + k)$-connected.
\end{lemma}

\begin{proof}
 We look at the connectedness of the pair $(X_d^k,X^k_{(d-3),d})$, which is obtained by attaching cells corresponding to $A\in \mathbf{S}_c(d-2,k)-\mathbf{S}^{(d-3)}(d-2,k)$. There is only one cell $D_A$ which has dimension at most $(d-1)(d-2)/2 + k$, namely the one corresponding to
\begin{eqnarray*}
 A&=&\left(\begin{array}{ccccccc}
            0 & \cdots & 0 & + & \ast & \cdots & \ast \\
            \vdots & & & \ddots & \ddots & \ddots & \vdots \\
            \vdots & & & & \ddots & + & \ast\\
            0 & & \cdots & \cdots & & 0 & \ast
           \end{array}
\right),
\end{eqnarray*}
and if $k>d-2$, there is only one cell with dimension $(d-1)(d-2)/2 + k+1$, namely the one corresponding to
\begin{eqnarray*}
 A'&=&\left(\begin{array}{ccccccccc}
            0 & \cdots & 0 & + & \ast & \ast & \cdots & \cdots & \ast \\
            \vdots & & & 0 & 0 & + & \ddots & & \vdots \\
            \vdots & & & & & \ddots & \ddots & \ddots & \vdots \\
            \vdots & & & & & & \ddots & + & \ast\\
            0 & & & \cdots & \cdots &  & & 0 & \ast
           \end{array}
\right).
\end{eqnarray*}
With the discussion on boundaries given above, we get $H_m(X_d^k,X^k_{(d-3),d};\Z/2\Z)= \Z/2\Z$ for $m=(d-1)(d-2)/2 + k$. Using Corollary \ref{collapse}, the result follows.
\end{proof}

\begin{theorem}\label{detectsphere}
  Let $\ell\in \R^n$ be a generic length vector with $a_0(\ell)=1$, and $d\geq 3$. Then $\mathcal{M}_d(\ell)$ is $((d-1)(d-2)/2+d-3)$-connected. Furthermore $\mathcal{M}_{n-1}(\ell)$ is homotopy equivalent to the sphere of dimension $n(n-3)/2$.
\end{theorem}

\begin{proof}
 The proof begins in the same way as the proof of Proposition \ref{contractible}, with the filtration
\[
 \emptyset = \mathcal{M}^0 \subset \mathcal{M}^1 \subset \cdots \subset \mathcal{M}^m = \mathcal{M}_d(\ell).
\]
As long as the index of the critical point is $k(d-1)$ with $k<d-2$, no new homology occurs, but if $k\geq d-2$, new cells may arise. However, by Lemma \ref{connectedness} the new $\mathcal{M}^{i+1}$ is still $((d-1)(d-2)/2+d-3)$-connected.

If $d=n-1$, the case $k\geq d-2=n-3$ only appears once, with the absolute maximum of the function. In that case only one cell of dimension $(d-1)(d-2)/2 + d-2$ is attached to a contractible space. Hence, up to homotopy, we get a sphere of dimension $(n-2)(n-3)/2 + n-3=n(n-3)/2$.
\end{proof}

As mentioned in the introduction, the last result can be improved to a homeomorphism between $\mathcal{M}_{n-1}(\ell)$ and the sphere. To see this, note that the closure of the space $\Omega_{n-1}$ of \cite{schoen} can be identified with $\mathcal{M}_n(\ell)$ for $\ell \in \R^n$ by sending a linkage configuration to the distances between the points. By \cite[Thm.1]{schoen}, this space is homeomorphic to a disc of dimension $n(n-3)/2$, and the boundary points correspond to those points $x\in \mathcal{M}_n(\ell)$ whose rank is at most $n-2$.

The space $\mathcal{M}_{n-1}(\ell)$ is now obtained by doubling $\mathcal{M}_n(\ell)$ along the boundary, compare Lemma \ref{easylemma} and also the proof of \cite[Thm.C]{kamiyb}.

\section{Basic homological properties of $(X_d^k,\partial X_d^k)$}
\label{sec8}

Let us denote the cellular chain complex for the pair $(X_d^k,\partial X_d^k)$ by $C_\ast$, freely generated by the matrices of $\mathbf{S}_c(d-2,k)$. This contains the subcomplex for the pair $(X^k_{(d-3),d},\partial X_d^k)$, which we denote by $C^0_\ast$, and which is freely generated by the matrices of $\mathbf{S}^{(d-3)}(d-2,k)$. By Lemma \ref{collapse} $H_\ast(C^0_\ast)=0$, and hence
\begin{eqnarray*}
 H_\ast(X_d^k,\partial X_d^k) &=& H_\ast(D_\ast),
\end{eqnarray*}
where $D_\ast = C_\ast/C_\ast^0$ is freely generated by matrices whose last two rows are of the form (\ref{fulcel}), and where the last row is non-zero.

Let us assume that $k\geq d-2$, so that $D_\ast\not=0$.

Notice that we can write $D_\ast$ as a direct sum of chain complexes
\begin{eqnarray}\label{splitchains}
 D_\ast&=&\bigoplus_{j=1}^{k-d+3} D^k_\ast(j),
\end{eqnarray}
where $D^k_\ast(j)$ is generated by those matrices which have $(k-d+4)-j$ column containing just $\ast$. In particular, $D^k_\ast(1)$ has only one generator, corresponding to the matrix
\begin{eqnarray*}
 A&=&\left(\begin{array}{cccccc}
            + & \ast & \cdots & \ast & \cdots & \ast \\
            0 & \ddots & \ddots & \vdots & & \vdots \\
            \vdots & \ddots & + & \ast & \cdots & \ast \\
            0 & \cdots & 0 & \ast & \cdots & \ast
           \end{array}
\right)
\end{eqnarray*}
while $D^k_\ast(k-d+3)$ has the most generators. The dimension of the cell $D_A$ is therefore
\begin{eqnarray*}
 k+(k-d+3)(d-2)+\frac{(d-2)(d-3)}{2}&=& k(d-1)-\frac{(d-2)(d-3)}{2}.
\end{eqnarray*}
The top-dimensional cell in $D^k_\ast(j)$ corresponds to a matrix of the form
\[
 \left(\begin{array}{cccccccccc}
            + & \ast & \cdots & \ast & \cdots& \cdots & \ast & \ast & \cdots & \ast \\
            0 & \ddots & \ddots & \vdots & & & \vdots& \vdots & & \vdots \\
            \vdots & \ddots & + & \ast & \cdots & \cdots & \ast & \vdots & & \vdots \\
            0 & \cdots & 0 & 0 & \cdots & 0 & + & \ast & \cdots & \ast \\
            0 & \cdots & & & & \cdots & 0 & \ast & \cdots & \ast
           \end{array}
\right)
\]
while the minimal-dimensional cell corresponds to a matrix of the form
\[
 \left(\begin{array}{ccccccccc}
            0 & \cdots & 0 & + & \ast & \cdots & \ast & \cdots & \ast \\
            \vdots & & \vdots & 0 & \ddots & \ddots & \vdots & & \vdots \\
            \vdots & & \vdots & \vdots & \ddots & + & \ast & \cdots & \ast \\
            0 & \cdots & 0 & 0 & \cdots & 0 & \ast & \cdots & \ast
           \end{array}
\right)
\]
So all the cells in $D_\ast^k(j)$ have dimension between
\[
 k(d-1)-\frac{(d-2)(d-3)}{2}-(j-1)(d-2)\mbox{ and } k(d-1)-\frac{(d-2)(d-3)}{2}-2(j-1).
\]

If we consider the complexes $D^k_\ast(j)$ with coefficients in $\Z/2\Z$, we get that every boundary is zero. This follows from the discussion at the beginning of Section \ref{boundarysection}, as there always is a column containing only $\ast$.

Even with coefficients in $\Z$ we can obtain some basic results on the homology of $\mathcal{M}_d(\ell)$.

\begin{proposition}\label{tophomology}
 Let $\ell\in \R^n$ be a generic length vector with $a_0(\ell)=1$, let $d\geq 4$ and let $n\geq d+1$. Then
\[
 H_{\mathbf{d}^n_d}(\mathcal{M}_d(\ell);\Z) \,\,\, = \,\,\, \Z \hspace{0.4cm}\mbox{and}\hspace{0.4cm}
 H_{\mathbf{d}^n_d-1}(\mathcal{M}_d(\ell);\Z) \,\,\, = \,\,\, 0
\]
Recall that $\mathbf{d}^n_d$ denotes the dimension of $\mathcal{M}_d(\ell)$.
\end{proposition}

\begin{proof}
  Let $\tilde{F}:\mathcal{C}_d(\ell)\to \R$ be the $SO(d-1)$-invariant Morse function of Proposition \ref{perfectmorse}, $F:\mathcal{M}_d(\ell)\to \R$ the induced function and let
\[
 \emptyset = \mathcal{M}^0 \subset \mathcal{M}^1 \subset \cdots \subset \mathcal{M}^m = \mathcal{M}_d(\ell)
\]
be a filtration such that $\mathcal{M}^m= F^{-1}((-\infty,a_m])$ for some sequence of regular values of $\tilde{F}$ such that $\mathcal{M}^m-\mathcal{M}^{m-1}$ contains exactly one critical point.

Notice that $\tilde{F}$ has only one critical manifold of index $(n-3)(d-1)$, which is the absolute maximum. Since $\mathcal{M}^{m-1}$ has the homotopy type of a CW-complex with lower dimensional cells, we get $H_q(\mathcal{M}^{m-1};\Z)=0$ for $q\geq \mathbf{d}^n_d-1$. Now $\mathcal{M}_d(\ell)$ is, up to homotopy, obtained from $\mathcal{M}^{m-1}$ by attaching the cells from $((X_d^{n-3},\partial X_d^{n-3})$. Only one cell has dimension $\geq \mathbf{d}^n_d-1$, and this cell has dimension $\mathbf{d}^n_d$. The result follows.
\end{proof}

\begin{theorem}
 Let $\ell\in \R^n$ be a generic length vector with $a_0(\ell)=1$, let $d\geq 4$ and let $n\geq d+2$. Then $\mathcal{M}_d(\ell)$ does not satisfy Poincar\'e duality with coefficients in $\Z/2\Z$. In particular, $\mathcal{M}_d(\ell)$ is not a topological manifold, with or without boundary.
\end{theorem}

\begin{proof}
First notice that $\mathcal{M}_d(\ell)$ cannot be a manifold with non-empty boundary, as by Proposition \ref{tophomology} $H_{\mathbf{d}^n_d}(\mathcal{M}_d(\ell))\not=0$.

Let us use the same filtration as in the previous proof. 

We will distinguish the cases $d=4$ and $d\geq 5$. Let us first assume that $d\geq 5$. Then $\mathcal{M}^{m-1}$ has the homotopy type of a CW-complex with cells of dimension at most $(n-4)(d-1)-\frac{(d-2)(d-3)}{2}$.

As $n\geq d+2$, we get that $H_{\mathbf{d}^n_d-2}(D^{n-3}(2);\Z/2\Z)=\Z/2\Z$, which corresponds to the maximal cell for $D^{n-3}(2)$. As $d\geq 5$, we get $\mathbf{d}^n_d-2-((n-4)(d-1)-\frac{(d-2)(d-3)}{2})\geq 2$, so
\begin{eqnarray*}
 H_{\mathbf{d}^n_d-2}(\mathcal{M}_d(\ell);\Z/2\Z)& \cong & \Z/2\Z.
\end{eqnarray*}
But $H_2(\mathcal{M}_d(\ell);\Z/2\Z)=0$ by Theorem \ref{detectsphere}, so Poincar\'e duality cannot hold.

Now consider the case $d=4$. Since $\mathbf{d}^n_d-2-((n-4)(d-1)-\frac{(d-2)(d-3)}{2})=1$ now, it is not clear whether $H_{\mathbf{d}^n_d-2}(\mathcal{M}_d(\ell);\Z/2\Z)\not=0$.

But let $c$ be the number of critical manifolds of index $3(n-4)$. By Proposition \ref{threeBetti}, we get $c=1+a_1(\ell)-a_{n-3}$. Now $a_{n-3}(\ell)\leq 1$, and if $a_{n-3}(\ell)=1$, then $a_1(\ell)=n-3$. As $n\geq 6$, we get $c\geq 2$, unless $a_1(\ell)=0$ in which case $c=1$.

Let us first consider the case $c\geq 2$. In that case the top-dimensional non-zero homology group of $\mathcal{M}^{m-1}$ is in degree $3(n-4)-1$, and the rank of this homology group is $c$. Attaching one cell of dimension $3(n-4)$ cannot kill this homology group, therefore $H_{3(n-4)-1}(\mathcal{M}_4(\ell);\Z/2\Z)\not =0$. But by Theorem \ref{detectsphere} we have $H_3(\mathcal{M}_4(\ell);\Z/2\Z)=0$, so Poincar\'e duality cannot hold.

It remains to consider the case $c=1$. In that case $\mathcal{S}^m_1(\ell)=\emptyset$ (where $m$ is chosen so that $\ell$ is $m$-dominated), which uniquely determines the chamber of $\ell$. In fact, we can assume that
\begin{eqnarray*}
 \ell&=&(1,\ldots,1,n-2).
\end{eqnarray*}
By Proposition \ref{theshapespace} we get that $\mathcal{M}_4(\ell)\approx \Sigma^{n-1}_3$, the shape space defined in the appendix. But this space is known to not satisfy Poincar\'e duality, see \cite[\S 4,\S 5]{kebacl}. In fact, the homology calculations in \cite{kebacl} give the same contradiction as above.
\end{proof}

\section{Homology of $(X_d^k,\partial X_d^k)$}
\label{sec9}

In this section we want to improve on the homology calculations of $H_\ast(X_d^k,\partial X_d^k)$. Let us begin with the case $d=4$. In that case
\begin{eqnarray*}
 D_\ast^k(j)&=& (\Z,3k-1-2(j-1)),
\end{eqnarray*}
where we use the notation $(G,n)$ for the graded group whose only non-zero degree is $n\in \Z$, in which case the entry is the abelian group $G$.

It follows that for $k\geq 1$ we get
\begin{eqnarray*}
 H_\ast(X_4^k,\partial X_4^k)&=&\bigoplus_{j=1}^{k-1}\, (\Z,3k-1-2(j-1)).
\end{eqnarray*}

In the case $d\geq 5$ we have to analyze the boundary operator more carefully. This is done by following the methods of \cite[\S 4]{kebacl}. Let us take a closer look at $d=5$. The matrices appearing for the generators of $D_\ast^k(j)$ are of the form
\[
 \left(
\begin{array}{ccccccccccc}
 0 & \cdots & 0 & + & \ast & \cdots & \ast & \ast & \ast & \cdots & \ast \\
 0 & \cdots & 0 & 0 & \cdots & \cdots & 0 & + & \ast & \cdots & \ast \\
 0 & & & \cdots & & & 0 & 0 & \ast & \cdots & \ast
\end{array}
\right)
\]
A typical boundary is of the form
\begin{eqnarray*}
 \partial \left(
\begin{array}{cccc}
 + & \ast & \ast & \ast \\
 0 & 0 & + & \ast \\
 0 & 0 & 0 & \ast 
\end{array}
\right)&=&
\left(
\begin{array}{cccc}
 0 & + & \ast & \ast \\
 0 & 0 & + & \ast \\
 0 & 0 & 0 & \ast 
\end{array}
\right)+ \left(
\begin{array}{cccc}
 0 & - & \ast & \ast \\
 0 & 0 & + & \ast \\
 0 & 0 & 0 & \ast 
\end{array}
\right)
\end{eqnarray*}
The second matrix on the right-hand side comes from the fact that we write $D^2= D^2_+\cup D^2_-$, but our symbolic matrices require a $+$ and not a $-$. Let $A$ be the diagonal matrix which has $-1$ in the second and fourth entry, and $1$ in the first and third entry. Then $A(D_-^2)= D_+^2$. This means we get the same matrix on the right side twice. To work out the exact coefficients, we need to take a closer look at orientations.

Recall that the matrices stand for products of discs $D^i$ or $D^i_+$, and every non-zero entry corresponds to one dimension. To choose an orientation, we choose the standard orientation of the discs $D^i$. We can actually think of every non-zero entry in the matrix coming with a basis vector into that dimension, and by picking an order of the entries in the matrix we get the orientation.

Let us go back to the example above. The matrix $A\in SO(3)$ used to turn $D_-^2$ into $D_+^2$ changes the orientation of $D_+^2$. It also changes the orientation of the next factor, which is a $D^3_+$. But for the final factor $D^4$, two basis elements are changed, so there is no impact on the orientation. Since we had two changes of orientations, we see that
\begin{eqnarray*}
 \partial \left(
\begin{array}{cccc}
 + & \ast & \ast & \ast \\
 0 & 0 & + & \ast \\
 0 & 0 & 0 & \ast 
\end{array}
\right)&=&\pm 2
\left(
\begin{array}{cccc}
 0 & + & \ast & \ast \\
 0 & 0 & + & \ast \\
 0 & 0 & 0 & \ast 
\end{array}
\right)
\end{eqnarray*}
Notice that adding extra factors of $D^4$ to the right has no impact on the signs. Adding extra factors of $D^2$ however does change the signs. So to work out the homology of $D_\ast^k(j)$, we can ignore the last two rows of the matrix.

The remaining $d-2$ rows all have to start with a $+$. Let us describe the remaining matrices using sequences of decreasing numbers, compare \cite[\S 4]{kebacl}. We can encode the matrix by a sequence of numbers $(k_1,\ldots,k_m)$ with
\[
 k_1\,\,\,>\,\,\,k_2\,\,\,>\,\,\,\cdots \,\,\,>\,\,\, k_m \geq 1
\]
where each number $k_i$ stands for the number of non-zero entries in the $i$-th row.

Let $E_\ast(m,j)$ be the chain complex freely generated by such sequences $(k_1,\ldots,k_m)$ where $k_1\leq m+j-1$, and we say that the sequence $(k_1,\ldots,k_m)$ has degree $k_1+\cdots +k_m -m(m+1)/2$. The boundary is given by
\begin{eqnarray*}
 \partial (k_1,\ldots,k_m)&=& \sum_{j=1}^m (-1)^{k_1+\cdots + k_{j-1}}(1+(-1)^{k_j})(k_1,\ldots,k_j-1,\ldots,k_m).
\end{eqnarray*}
where a sequence $(k_1',\ldots,k_m')$ is interpreted as $0$ if $k_i'=k_{i+1}'$ for some $i\in 1,\ldots,m-1$ or if $k_m=0$.

\begin{remark}
 The sign $(-1)^{k_1+\cdots + k_{j-1}}$ comes from the following: Each non-zero entry in the symbolic matrix spans a dimension, but only the entries with a $+$ have a non-zero boundary. If we order the basis for the orientation by starting with the first row on the left, the $+$ is at the $0$-th position. Similarly, the $+$ in the second row is in the $k_1$-th position, and so on.
\end{remark}

It follows that, possibly up to a sign which has no impact on the homology,
\begin{eqnarray}\label{introe}
 D_\ast^k(j)&=& E_{\ast-u}(d-4,j),
\end{eqnarray}
where $u=(d-1)k-(d-2)(d-3)/2-(d-2)(j-1)$.

Let us take a look at the case $d=5$. Then $E_\ast(1,j)$ is generated by $(k)$, where $k\leq j$, the boundary maps are alternating between $2$ and $0$, and $\partial(2)=2(1)$. We thus get

\begin{lemma}\label{hom(e(1,j))}
 Let $j\geq 1$. Then
\begin{eqnarray*}
 H_q(E_\ast(1,j))&=& \left\{
\begin{array}{cl}
 \Z & j \mbox{ odd, }q=j-1\\
 0 & q \mbox{ odd, or }q\geq j\\
 \Z/2\Z & q \mbox{ even}
\end{array}
\right. .
\end{eqnarray*}

\end{lemma}

To understand the homology of $E_\ast(m,j)$ for $m\geq 2$, notice that we can think of this complex as the total complex of a double complex $E_{\ast\,\ast}$, where the horizontal grading measures the first row, and the vertical grading the remaining rows. We can therefore think of $E_\ast(m,j)$ as the total complex of the sequence of chain complexes with chain maps
\begin{equation}\label{seqchaincomplexes}
 E_\ast(m-1,1) \stackrel{1+(-1)^{m-1}}{\longleftarrow} E_\ast(m-1,2) \stackrel{1+(-1)^{m-2}}{\longleftarrow} \cdots
 \stackrel{1+(-1)^{m-j+1}}{\longleftarrow} E_\ast(m-1,j).
\end{equation}
Notice that every second map is $0$, so that the total complex is just a direct sum of sequences
\begin{eqnarray*}
 E_\ast(m-1,i) &\stackrel{2}{\longleftarrow} & E_\ast(m-1,i+1).
\end{eqnarray*}
Using this and the particular form of the boundary in $E_\ast(m,j)$ one can show that the homology of $E_\ast(m,j)$ only contains direct summands of $\Z$ and $\Z/2\Z$. One should compare this with the results in 
\cite[\S 4,\S 5]{kebacl}, where closed formulas for the number of such summands in the homology of similar chain complexes are given. As these closed formulas are not that enlightening, and since we need to enter the homology of $(X^k_d,\partial X^k_{d-1})$ into another spectral sequence coming from the filtration $(\mathcal{M}^k)_{k\geq 0}$, we will abandon torsion and look instead at homology with coefficients in $\Q$.

Lemma \ref{hom(e(1,j))} then reduces to
\begin{eqnarray*}
 H_\ast(E_\ast(1,j);\Q)&=& \left\{
\begin{array}{cl}
 0 & j \mbox{ even}\\
 (\Q,j-1) & j \mbox{ odd}
\end{array}
\right. .
\end{eqnarray*}

To describe the rational homology of $E_\ast(m,j)$ for $m\geq 2$ we want to give concrete generators, and then show that they span the homology. Let us begin with $m$ even, that is, $m=2n$ for some $n\geq 1$.

Let $j+m-1\geq k_1>\cdots > k_n\geq 2$ be a sequence of even numbers. Then 
\[
(k_1,k_1-1,k_2,k_2-1,\ldots,k_n,k_n-1) 
\]
is easily seen to be a cycle in $E(2n,j)$. Furthermore, no non-zero integer multiple can be a boundary, as only sequences which have a term $(k_i+1,k_i-1)$ in them could have this sequence in their boundary. But since $k_i+1$ is odd, the boundary formula has a factor $1+(-1)^{k_i+1}=0$. In particular, such cycles span a factor of $\Z$ in $H_\ast(E_\ast(2n,j))$. It is also easy to see that the degree of this cycle is a multiple of $4$.

For $m=2n+1$ we can look at the sequence
\[
 (k_0,k_1,k_1-1,k_2,k_2-1,\ldots,k_n,k_n-1)
\]
where the $k_i$ are as before for $i\geq 1$, and $j+m-1\geq k_0>k_1$. For this to be a cycle, we need $k_0$ to be odd. But if $k_0+1\leq j+m-1$, we get this to be a (rational) boundary. To obtain a $\Z$ factor in $H_\ast(E_\ast(2n+1,j))$, we therefore need $k_0=j+m-1$. As $m$ is odd, this is only possible if $j$ is odd. In this case, notice that the degree of this cycle is $j-1+4i$ for some $i\geq 0$.

\begin{proposition}\label{generators}
 Let $n\geq 1$ and $j\geq 1$. Then $H_\ast(E_\ast(2n,j);\Q)$ has a basis given by elements
\[
(k_1,k_1-1,k_2,k_2-1,\ldots,k_m,k_m-1) 
\]
where $k_i= 2(n+1-i)+j_i$ for $i=1,\ldots,n$, where $j_1\geq j_2 \geq \cdots\geq j_n\geq 0$ is a sequence of even numbers with $j_1+2n \leq j+2n-1$. The degree of $(k_1,k_1-1,k_2,k_2-1,\ldots,k_n,k_n-1)$ is $2(j_1+\cdots+j_n)$.

Furthermore, $H_\ast(E_\ast(2n+1,j);\Q)=0$ for $j$ even, and $H_\ast(E_\ast(2n+1,j);\Q)$ for $j$ odd has a basis given by elements
\[
 (k_0,k_1,k_1-1,k_2,k_2-1,\ldots,k_n,k_n-1)
\]
with the $k_i$ as above, and $k_0=j+2n$. The degree of $(k_0,k_1,k_1-1,k_2,k_2-1,\ldots,k_n,k_n-1)$ is $(j-1)+2(j_1+\cdots+j_n)$.
\end{proposition}

\begin{proof}
 The proof is by induction. Let us first show that the statement for $2n-1$ implies the statement for $2n$. We get that the chain complex $E_\ast(2n,j)$ is the total complex of the sequence (\ref{seqchaincomplexes}). Now all the chain complexes $E_\ast(2n-1,i)$ with $i$ even have $0$ as their homology, so the homology of $E_\ast(2n,j)$ is the direct sum of the homologies of $E_\ast(2m-1,i)$ with $i$ odd. The basis elements for $H_\ast(E_\ast(2m-1,i);\Q)$ are then of the form
\[
 (i+2n-2,k_1,k_1-1,\ldots,k_{n-1},k_{n-1}-1)
\]
with $i\leq j$ odd by the induction assumption. The way we think of $E_\ast(2n,j)$ as a double complex means these generators correspond to
\[
 (i+2n-1,i+2n-2,k_1,k_1-1,\ldots,k_{n-1},k_{n-1}-1).
\]
But this gives exactly the statement for the rational homology of $E_\ast(2n,j)$. Notice that this also works for $n=1$.

It remains to show that the statement for $2n$ implies the statement for $2n+1$.

Again we use the sequence (\ref{seqchaincomplexes}). The condition that $k_1\leq i+2n-1$, implies that for $i$ odd the homologies of $E_\ast(2n,i)$ and $E_\ast(2n,i+1)$ have the same basis. Furthermore, in (\ref{seqchaincomplexes}) we get for $i$ odd terms of the form
\begin{eqnarray*}
 E_\ast(2n,i) & \stackrel{2}{\longleftarrow} & E_\ast(2n,i+1)
\end{eqnarray*}
which induce isomorphisms on rational homology. In particular, for $j$ even all homology vanishes. For $j$ odd we are left with
\begin{eqnarray*}
 H_\ast(E_\ast(2n+1,j);\Q)&\cong & H_{\ast-j}(E_\ast(2n,j);\Q),
\end{eqnarray*}
and because of the way the double complex structure of $E_\ast(2n+1,j)$ is formed, we see that the basis is represented by elements
\[
 (j+2n,k_1,k_1-1,k_2,k_2-1,\ldots,k_n,k_n-1).
\]
The statement about the degrees of these basis elements is easy to see.
\end{proof}

\begin{definition}
 Let $m\geq 1$. Then define
\begin{eqnarray*}
 \nabla_m&=&\{(j_1,\ldots,j_m)\in \Z^m\,|\, j_1\geq j_2\geq \cdots\geq j_m\geq 0\}.
\end{eqnarray*}
Also, if $(j_1,\ldots,j_m)\in \nabla_m$, we define
\begin{eqnarray*}
 |(j_1,\ldots,j_m)|&=&j_1+j_2+\cdots +j_m
\end{eqnarray*}
and
\begin{eqnarray*}
 \|(j_1,\ldots,j_m)\|&=& 2j_1+1.
\end{eqnarray*}
Also, for $m=0$ we let $\nabla_0=\{()\}$, where we think of $()$ as a point with $|()|=0$ and $\|()\|=1$.
\end{definition}

We write elements of $\nabla_m$ as $\mathbf{j}=(j_1,\ldots,j_m)$. Notice that every $\mathbf{j}\in \nabla_m$ produces a generator in the homology of $E(2m,j)$, provided that $j\geq \|\mathbf{j}\|$, whose degree is $4|\mathbf{j}|$, and a generator in the homology of $E(2m+1,2i+1)$, provided that $2i+1\geq \|\mathbf{j}\|$, whose degree is $2i+4|\mathbf{j}|$.

For $j\geq 1$ and $m\geq 0$ let
\begin{eqnarray*}
 \nabla_m(j)&=&\{\mathbf{j}\in \nabla_m\,|\, \|\mathbf{j}\|\leq j\}.
\end{eqnarray*}

We denote the Poincar\'e polynomial of the pair $(X^k_d,\partial X^k_d)$ by $P_d^k(t)$, so:
\begin{eqnarray*}
 P_d^k(t)&=&b_0(X^k_d,\partial X^k_d)+ b_1(X^k_d,\partial X^k_d)\, t + \cdots b_n(X^k_d,\partial X^k_d) \, t^n
\end{eqnarray*}
where $n$ is the dimension of $X^k_d$ and $b_j(X^k_d,\partial X^k_d)$ is the $\Z$-rank of $H_j(X^k_d,\partial X^k_d)$ for all $j=0\,\ldots,n$.

\begin{theorem}\label{homologyofx}
 Let $m\geq 0$ and $k\geq 1$. Then
\begin{eqnarray*}
 P^k_{2m+4}(t)&=&t^{k+(m+1)(2m+3)}\sum_{\mathbf{j}\in \nabla_m(k-2m-1)}t^{4|\mathbf{j}|}\,\frac{t^{(2m+2)(k-2m-\|\mathbf{j}\|)}-1}{t^{2m+2}-1}
\end{eqnarray*}
and
\begin{eqnarray*}
 P^k_{2m+5}(t)&=& t^{u(k,m)}\,\sum_{\mathbf{j}\in \nabla_m(k-2m-2)} t^{4|\mathbf{j}|}\, \frac{t^{4(m+1)\lfloor \frac{k-2m-\|\mathbf{j}\|}{2}\rfloor}-1}{t^{4(m+1)}-1}
\end{eqnarray*}
where
\begin{eqnarray*}
 u(k,m)&=&(2m+4)k-(2m+3)(m+1)-4(m+1)\lfloor\frac{k-2m-3}{2}\rfloor.
\end{eqnarray*}

\end{theorem}

Here $\lfloor x\rfloor = \max\{n\in \Z \,|\, n\leq x\}$.

\begin{proof}
 The proof is now merely an organisation of our previous results, using $d=2m+4$ or $2m+5$. By (\ref{splitchains}) the homology of $(X^k_d,\partial X^k_d)$ splits into summands, which by (\ref{introe}) come from $E_\ast(2m,j)$ or $E_\ast(2m+1,j)$ shifted by 
\begin{eqnarray*}
v(d,k,j) &=& (d-1)k-(d-2)(d-3)/2-(d-2)(j-1),
\end{eqnarray*}
 and where $j=1,\ldots,k-d+3$. 

Let $d=2m+4$. Using Proposition \ref{generators} we see that each $\mathbf{j}\in \nabla_m(k-2m-1)$ produces a homology generator, and in fact for each $j=1,\ldots,k-2m-1$ with $j\geq \|\mathbf{j}\|$. The degree of such a generator is $4|\mathbf{j}|+v(2m+4,k,j)$, so the degrees vary from
\[
 4|\mathbf{j}|+(2m+3)k-(m+1)(2m+1)-2(m+1)(\|\mathbf{j}\|-1)
\]
down to
\begin{eqnarray*}
 4|\mathbf{j}|+v(2m+4,k,k-2m-1)&=&4|\mathbf{j}|+k+(m+1)(2m+3)
\end{eqnarray*}
in steps of $2(m+1)$. Using
\begin{eqnarray*}
 1+t^{2m+2}+\cdots t^{(2m+2)(k-2m-1-\|j\|)}&=& \frac{t^{(2m+2)(k-2m-\|\mathbf{j}\|)}-1}{t^{2m+2}-1},
\end{eqnarray*}
we get the result.

The case $d=2m+5$ is very similar, each $\mathbf{j}\in \nabla_m(k-2m-1)$ produces a homology generator, but only for each odd $j=1,\ldots,k-2m-2$ with $j\geq \|\mathbf{j}\|$, and with degree $(j-1)+4|\mathbf{j}|+v(2m+5,k,j)$. A similar argument to the even case gives the stated result. Note that $2\lfloor \frac{n-1}{2} \rfloor + 1$ is the largest odd number not bigger than $n$, and the degree increase for each $\mathbf{j}$ is $4(m+1)$ because we only consider odd numbers between $\|\mathbf{j}\|$ and $k-2m-2$.
\end{proof}

For small values of $m$ the sets $\nabla_m$ have a very simple form, so we collect the Poincar\'e polynomials in these special cases in the next corollary.

\begin{corollary}\label{someformulas}
 For $k\geq 2$ we have
\begin{eqnarray*}
 P^k_4(t)&=& t^{k+3}\,\sum_{i=0}^{k-2}t^{2i}.
\end{eqnarray*}
For $k\geq 3$ we have
\begin{eqnarray*}
 P^k_5(t)&=& t^{4k-4\lfloor\frac{k-3}{2}\rfloor-3}\,\sum_{i=0}^{\lfloor\frac{k-3}{2}\rfloor}t^{4i}.
\end{eqnarray*}
For $k\geq 4$ we have
\begin{eqnarray*}
 P^k_6(t) &=& t^{k+10} \, Q_{k-4}(t^4).
\end{eqnarray*}
For $k\geq 5$ we have
\begin{eqnarray*}
 P^k_7(t) &=& t^{6k-8\lfloor\frac{k-5}{2}\rfloor-10} \, (Q_{\lfloor\frac{k-5}{2}\rfloor}(t^8)+t^4\,Q_{\lfloor\frac{k-5}{2}\rfloor-1}(t^8)).
\end{eqnarray*}

\end{corollary}

\begin{proof}
 The cases with $m=0$ are easy to see, as $\nabla_0$ only consists of one element.

To determine $P^k_6(t)$ and $P^k_7(t)$, note that
\begin{eqnarray*}
 \nabla_1(n)&=&\left\{(i)\,\left|\, 0 \leq i \leq \lfloor \frac{n-1}{2} \rfloor \right\} \right. .
\end{eqnarray*}
Therefore
\begin{eqnarray*}
 P_6^k(t)&=& t^{k+10} \sum_{i=0}^{\lfloor\frac{k-4}{2}\rfloor}t^{4i}\, \frac{t^{4(k-3-2i)}-1}{t^4-1}\\
&=& t^{k+10} \, Q_{k-4}(t^4)
\end{eqnarray*}
by Lemma \ref{qsumformulas}. Similarly,
\begin{eqnarray*}
 P_7(t)&=& t^{u(k,1)} \sum_{i=0}^{\lfloor\frac{k-5}{2}\rfloor}t^{4i}\, 
\frac{t^{8(\lfloor\frac{k-5}{2}\rfloor+1-i)}-1}{t^8-1}\\
&=& t^{6k-8\lfloor\frac{k-5}{2}\rfloor-10} (Q_{\lfloor\frac{k-5}{2}\rfloor}(t^8)+t^4\,Q_{\lfloor\frac{k-5}{2}\rfloor-1}(t^8))
\end{eqnarray*}
by Lemma \ref{qsumformulasx}.
\end{proof}

\section{Poincar\'e polynomials for linkage spaces in odd dimensional Euclidean spaces}
\label{sec10}

In order to calculate the Poincar\'e polynomial of $\mathcal{M}_d(\ell)$ for $d\geq 4$, we want to take the filtration
\[
 \emptyset = \mathcal{M}^0 \subset \mathcal{M}^1 \subset \cdots \subset \mathcal{M}^m = \mathcal{M}_d(\ell)
\]
which arises from the $SO(d-1)$-invariant Morse-Bott function of Proposition \ref{perfectmorse}, so that
\begin{eqnarray*}
 \mathcal{M}^{s+1}&\simeq& \mathcal{M}^s\cup_{\partial X^k_d}X^k_d
\end{eqnarray*}
for all $s=0,\ldots,m-1$ and appropriate $k$ depending on $s$. The long exact sequence of the pair $(\mathcal{M}^{s+1},\mathcal{M}^s)$ takes on the form
\[
 \cdots \longrightarrow H_{t+1}(X^k_d,\partial X^k_d) \longrightarrow H_t(\mathcal{M}^s) \longrightarrow H_t(\mathcal{M}^{s+1}) \longrightarrow H_t(X^k_d,\partial X^k_d) \longrightarrow \cdots
\]
If we look at $u(k,m)$ in Theorem \ref{homologyofx}, we see that for fixed $m$ this number is always odd (for $m$ even) or even (for $m$ odd) for all $k\geq 2m+3$. It follows that
\begin{eqnarray}\label{niceoddbehaviour}
 H_\ast(\mathcal{M}^{s+1};\Q)&=& H_\ast(\mathcal{M}^s;\Q) \oplus H_\ast(X^k_d,\partial X^k_d;\Q)
\end{eqnarray}
for odd $d\geq 5$.

This is not true for $d\geq 4$ even, as the following example shows.

\begin{example}
 Let $\ell^6=(1,1,1,1,1,4)$. Then $a_0(\ell^6)=1$ and $a_i(\ell^6)=0$ for all $i\geq 1$. Therefore the Morse numbers $\mu_i(\ell)$ of the Morse function $f_3:\mathcal{M}_3(\ell^6)\to \R$ of Corollary \ref{perfect3} are all $1$. As $\ell^+=(1,1,1,1,5)$ has empty moduli space, we can construct the Morse function so that the indices in the filtration are increasing. We thus have $4$ critical points of index $0$, $2$, $4$ and $6$, respectively, so the respective values for $k$ are $0$, $1$, $2$ and $3$.

If we look at the analogous function for $d=4$, the filtration satisfies
\[
 \mathcal{M}^1 \simeq \ast, \,\,\, \mathcal{M}^2 \simeq S^5, \,\,\, \mathcal{M}^3\,\,\,=\,\,\, \mathcal{M}_4(\ell^6) \simeq \mathcal{M}^2 \cup e^6 \cup e^8
\]
which means that up to homotopy $\mathcal{M}_4(\ell^6)$ is obtained from $S^5$ by adding a $6$-cell and an $8$-cell. By Proposition \ref{theshapespace}, we have $\mathcal{M}_4(\ell^6) = \Sigma^5_3$, and the $\Z$-homology of this space has been calculated in \cite[Table 5.3]{kebacl} as
\begin{eqnarray*}
 H_\ast(\Sigma^5_3) &=& \left \{ \begin{array}{cl}
                                  \Z & \ast = 8 \\
                                  \Z/2\Z & \ast = 5 \\
                                  0 & \mbox{else}
                                 \end{array}
\right. .
\end{eqnarray*}
This shows that there is a non-trivial interaction between the critical points of index $2$ and $3$, which persists when looking at $\ell^n=(1,\ldots,1,n-2)\in \R^n$, as \cite[Table 5.3]{kebacl} shows.
\end{example}

One would expect similar interactions when looking at more general $\ell$, but we leave that for a future project.

\begin{definition}
Let $\ell\in \R^n$ be a generic length vector, and $d\geq 2$. We denote the Poincar\'e polynomial of $\mathcal{M}_d(\ell)$ with $\Z$ coefficients by $P^\ell_d(t)$.
\end{definition}

The next proposition follows by a simple induction on (\ref{niceoddbehaviour}), using Proposition \ref{threeBetti}.

\begin{proposition}\label{simplepoin}
 Let $\ell\in \R^n$ be a generic length vector, and $d=2m+5$ with $m\geq 0$. Then
\begin{eqnarray*}
 P^\ell_d(t) &=& a_0(\ell) + \sum_{k=2m+3}^{n-3} \mu_k(\ell)\, P^k_d(t),
\end{eqnarray*}
where $\mu_k(\ell)$ are as in Proposition \ref{threeBetti}.
\end{proposition}

We can express the $\mu_k$ in terms of $a_k$, and the $P^k_d(t)$ are given by Theorem \ref{homologyofx}. Using Corollary \ref{someformulas}, we can make the dependence on the $a_k$ more explicit.

\begin{theorem}\label{polynomiald=5}
 Let $\ell\in \R^n$ be a generic length vector with $a_0(\ell)=1$. Let $m\geq 3$ be such that $n=2m-1$ or $n=2m$. Then
\begin{eqnarray*}
P_5^\ell(t)&=&1+ t^9\cdot\, \sum_{i=0}^{m-2}(a_i-a_{n-2-i})\,(Q_{n-6-i}(t^4)-Q_{i-4}(t^4)),
\end{eqnarray*}
where $a_i=a_i(\ell)$ for $i=1,\ldots,n-3$, and $a_{n-2}=0=Q_j$ for $j<0$.
\end{theorem}

\begin{proof}
 We know from Proposition \ref{threeBetti} that $a_0(\ell)$ contributes to each $\mu_k(\ell)$ for $k=0,\ldots,n-3$. Similarly, $a_1(\ell)-a_{n-3}(\ell)$ contributes to $\mu_1(\ell),\ldots,\mu_{n-4}(\ell)$, and $a_{m-2}-a_{n-m}$ contributes to $\mu_{m-2}$ and $\mu_{n-m-1}$.

Notice that $\mu_1$ and $\mu_2$ have no impact on the homology of $\mathcal{M}_5(\ell)$.

According to Corollary \ref{someformulas} the contribution of $a_0(\ell)=1$ to the Poincar\'e polynomial is therefore
\begin{eqnarray*}
 1+\sum_{k=3}^{n-3} P^k_5(t) &=& 1+ \sum_{k=3}^{n-3} \left(t^{4k-4\lfloor\frac{k-3}{2}\rfloor-3}\,\sum_{i=0}^{\lfloor\frac{k-3}{2}\rfloor}t^{4i}\right)\\
&=& 1+ t^9 \cdot \, \sum_{k=0}^{n-6} t^{4\lfloor \frac{k+1}{2}\rfloor}\,\frac{t^{4\lfloor \frac{k+2}{2}\rfloor}-1}{t^4-1}\\
&=&1+t^9\, Q_{n-6}(t^4),
\end{eqnarray*}
where we use Lemma \ref{qsumformulas} in the last line.

Similarly, the contribution of $a_1(\ell)-a_{n-3}(\ell)$ is $t^9\,Q_{n-7}(t^4)$, and so on. But notice that for $j\geq 4$ we get for the contribution of $a_j(\ell)-a_{n-2-j}(\ell)$ the formula
\begin{eqnarray*}
 t^9 \sum_{k=j-3}^{n-6-j} t^{4\lfloor \frac{k+1}{2}\rfloor}\,\frac{t^{4\lfloor \frac{k+2}{2}\rfloor}-1}{t^4-1}
 &=& t^9 \,(Q_{n-6-j}(t^4)-Q_{j-4}(t^4)).
\end{eqnarray*}
Since we set $Q_k(t)=0$ for negative $k$, this also holds for all $j\geq 0$. Adding all terms together gives the result.
\end{proof}

\begin{remark}
 If we write
\begin{eqnarray*}
 R_k(t) &=& 1 + t + \cdots + t^k,
\end{eqnarray*}
and $R_k(t)=0$ for $k<0$, we can describe the Poincar\'e polynomials of $\mathcal{M}_3(\ell)$ as
\begin{eqnarray*}
 P_3^\ell(t)&=& 1 + t^2 \cdot \, \sum_{i=0}^{m-2}(a_i-a_{n-2-i})\, (R_{n-4-i}(t^2)-R_{i-2}(t^2)),
\end{eqnarray*}
as follows easily from (\ref{poincare3}). Furthermore, we have $Q_{2m}(t)=R_m(t)R_m(t)$ and $Q_{2m+1}(t)=R_m(t)R_{m+1}(t)$. It is therefore natural to ask what the correct formula for $P_{2m+5}^\ell(t)$ is and whether it fits into a similar pattern. However, by looking at Corollary \ref{someformulas} in the case $d=7$, we see that the Poincar\'e polynomial of $\mathcal{M}_7(\ell)$ will have non-zero coefficients in even degrees between $26$ and $\mathbf{d}^n_7$ for $n\geq 9$.
\end{remark}

\begin{example}
 There exist $135$ chambers for $n=7$ up to permutations \cite{haurod}, and the Poincar\'e polynomial for $\ell\in \R^7$ is
\begin{eqnarray*}
 P^\ell_5(t) &=& a_0(\ell) (1+t^9 + t^{13}) + (a_1(\ell)-a_4(\ell)) t^9.
\end{eqnarray*}
Also, $a_4(\ell)=0$, unless $\ell=(1,1,1,1,5,5,5)$, in which case $a_1(\ell)-a_4(\ell)=3$. If we also assume that $\ell$ is different from $(1,1,1,1,1,1,7)$, the Poincar\'e polynomial is
\begin{eqnarray*}
 P^\ell_5(t)&=& 1 + (a_1(\ell)+1) t^9 + t^{13}.
\end{eqnarray*}
Since $a_1(\ell)\in \{0,\ldots,6\}$ there are not a lot of variations among the Poincar\'e polynomials. Also notice that $\mathcal{M}_5(\ell)$ up to homotopy is obtained from a wedge of $(a_1(\ell)+1)$ $9$-spheres by attaching a cell of dimension $10$, $11$ and $13$. As these three cells correspond to $a_0(\ell)$ it seems unlikely to expect too many different homotopy types between the chambers for $n=7$.
\end{example}

\section{The Euler characteristic for even dimensional linkage spaces}
\label{eulersec}

Let us begin with $\chi(X^k_{2m+4},\partial X^k_{2m+4})$. This is obtained by evaluating $P^k_{2m+4}(-1)$ in Theorem \ref{homologyofx}. Reorganising this term gives the following proposition.

\begin{proposition}
 Let $m\geq 0$ and $k\geq m+1$. Then
\begin{eqnarray*}
 \chi(X^{2k}_{2m+4},\partial X^{2k}_{2m+4})&=& (-1)^{m+1} \sum_{j_m=0}^{k-m-1} \sum_{j_{m-1}=j_m}^{k-m-1} \cdots \sum_{j_1=j_2}^{k-m-1} (2k-2m-2j_1-1),
\end{eqnarray*}
and
\begin{eqnarray*}
 \chi(X^{2k+1}_{2m+4},\partial X^{2k+1}_{2m+4})&=& (-1)^m \sum_{j_m=0}^{k-m-1} \sum_{j_{m-1}=j_m}^{k-m-1} \cdots \sum_{j_1=j_2}^{k-m-1} (2k-2m-2j_1).
\end{eqnarray*}

\end{proposition}

The simplest cases $m=0$ and $m=1$ are easily seen to give the following.

\begin{corollary}
 Let $m\geq 0$ and $k\geq m+1$. Then
\begin{eqnarray*}
 \chi(X^{2k}_{2m+4},\partial X^{2k}_{2m+4})+\chi(X^{2k+1}_{2m+4},\partial X^{2k+1}_{2m+4}) &=& (-1)^m |\nabla_m(2k-2m-1)|.
\end{eqnarray*}
Furthermore, for $k\geq 0$ we get
\begin{eqnarray*}
 \chi(X^k_4,\partial X^k_4)&=& (-1)^{k+1}(k-1)
\end{eqnarray*}
and
\begin{eqnarray*}
 \chi(X^{2k}_6,\partial X^{2k}_6)&=& (k-1)^2\\
\chi(X^{2k+1}_6,\partial X^{2k+1}_6)&=& -k(k-1).
\end{eqnarray*}

\end{corollary}

It is worth pointing out that $X^0_d$ is a point with empty boundary, so the Euler characteristic is just $1$.

\begin{corollary}\label{eulerd=4}
 Let $k\geq 3$ and $\ell\in \R^n$ be a generic length vector, and $n=2k$ or $n=2k-1$. Then for $n=2k$ we get
\begin{eqnarray*}
 \chi(\mathcal{M}_4(\ell))&=& \sum_{i=0}^{k-2} (-1)^i (a_i(\ell)-a_{2k-2-i}(\ell))\, (k-1-i)
\end{eqnarray*}
and for $n=2k-1$ we get
\begin{eqnarray*}
 \chi(\mathcal{M}_4(\ell))&=& -(k-3) \sum_{i=0}^{k-2} (-1)^i (a_i(\ell)-a_{2k-3-i}(\ell)).
\end{eqnarray*}

\end{corollary}

\begin{proof}
Assume that $n=2k-1$. Write $\chi_i=\chi(X^i_4,\partial X^i_4)=(-1)^{i+1}(i-1)$. Then $a_0(\ell)$ contributes $\chi_0+\chi_1+\ldots+\chi_{2k-4}$ to the Euler characteristic of $\mathcal{M}_4(\ell)$. Since $\chi_{2i}+\chi_{2i+1}=1$, this gives a contribution of
\begin{eqnarray*}
 \sum_{i=0}^{k-3} (\chi_{2i}+\chi_{2i+1})+\chi_{2k-4}&=&(k-2)-(2k-5)\\
&=&-(k-3).
\end{eqnarray*}
Similarly, the contribution of $(a_j(\ell)-a_{2k-3-j}(\ell)$ is $(-1)^{i+1}(k-3)$, where for odd $j$ one should note that $\chi_j+\chi_{j+1}=j-1$. Summing the contributions with the appropriate factor gives the result.

The result for $n=2k$ uses a similar argument.
\end{proof}

\begin{example}
 Let $\ell=(1,\ldots,1)\in \R^{2m+1}$. A subset $J\subset\{1,\ldots,2m+1\}$ is $\ell$-short if and only if has at most $m$ elements. It follows that $a_i=\binom{2m}{i}$ for $i=0,\ldots, m-1$ and $a_i=0$ for $i\geq m$. Hence
\begin{eqnarray*}
 \chi(\mathcal{M}_4(\ell)) &=& -(m-2) \sum_{i=0}^{m-1} (-1)^i \binom{2m}{i} \\
&=& (-1)^m (m-2) \binom{2m-1}{m-1}
\end{eqnarray*}
where we used $(-1)^{m-1}\binom{2m-1}{m-1} = \sum_{i=0}^{m-1} (-1)^i \binom{2m}{i}$ which follows from the binomial formula. This formula has been obtained by Kamiyama in \cite[Thm.A]{kamiya}.
\end{example}

\begin{corollary}\label{eulerd=6}
 Let $k\geq 4$ and $\ell\in \R^n$ be a generic length vector, and $n=2k$ or $n=2k-1$. Then for $n=2k$ we get
\begin{eqnarray*}
 \chi(\mathcal{M}_6(\ell))&=& \frac{1}{2}\sum_{i=0}^{k-2} (-1)^{i+1}\,c_i\,((k-3-\lfloor\!\!\!\begin{array}{c}\frac{i}{2}\end{array}\!\!\!\rfloor) (k-2-\lfloor\!\!\!\begin{array}{c}\frac{i}{2}\end{array}\!\!\!\rfloor)-(\lfloor \!\!\!\begin{array}{c} \frac{i-3}{2}\end{array}\!\!\!\rfloor \cdot \lfloor \!\!\!\begin{array}{c} \frac{i-1}{2}\end{array}\!\!\!\rfloor))
\end{eqnarray*}
and for $n=2k-1$ we get
\begin{eqnarray*}
 \chi(\mathcal{M}_6(\ell))&=& \frac{1}{2}\sum_{i=0}^{k-2}(-1)^i\,c_i\,((k-3-\lfloor\!\!\!\begin{array}{c}\frac{i+1}{2}\end{array}\!\!\!\rfloor)(k-2-\lfloor\!\!\!\begin{array}{c} \frac{i+1}{2}\end{array}\!\!\!\rfloor)+(\lfloor \!\!\!\begin{array}{c}\frac{i-3}{2}\end{array}\!\!\!\rfloor\cdot \lfloor \!\!\!\begin{array}{c}\frac{i-1}{2}\end{array}\!\!\!\rfloor)),
\end{eqnarray*}
where $c_i=a_i(\ell)-a_{n-2-i}(\ell)$.
\end{corollary}

\begin{proof}
 The proof is along the same lines as the proof of Corollary \ref{eulerd=4}. If $n=2k-1$, the contribution of $a_{2j}$ is
\begin{eqnarray*}
 \sum_{i=2j}^{2k-4-2j}\chi_i & =& \sum_{i=j}^{k-3-j}(\chi_{2i}+\chi_{2i+1}) + \chi_{2k-4-2j}\\
&=& -\sum_{i=j-1}^{k-4-j} i + (k-3-j)^2 \\
&=&\frac{1}{2}((k-3-j)(k-2-j)+(j-2)(j-1),
\end{eqnarray*}
and $a_{2j-1}$ contributes this term with a negative sign. The case $n=2k$ is similar.
\end{proof}

\begin{example}
 Let $\tilde{\ell}=(1,\ldots,1,7)\in \R^8$, so that $\mathcal{M}_d(\tilde{\ell})\cong \Sigma^7_{d-1}$ by Proposition \ref{theshapespace}. In particular, $a_0(\tilde{\ell})=1$ and $a_i(\tilde{\ell})=0$ for $i\geq 1$. Then
\[ 
\chi(\mathcal{M}_4(\tilde{\ell})) \,\,\,=\,\,\, 3 \hspace{1cm}\mbox{and}\hspace{1cm} \chi(\mathcal{M}_6(\tilde{\ell})) \,\,\,=\,\,\, 0
\]
as can be readily seen from Corollaries \ref{eulerd=4} and \ref{eulerd=6}.

If we let $\ell=(1,1,1,1,1,3,3,6)$, we see that $a_0(\ell)=1$, $a_1(\ell)=5$, $a_2(\ell)=10$ and $a_i(\ell)=0$ for $i\geq 3$. This implies
\begin{eqnarray*}
 \chi(\mathcal{M}_4(\ell))&=&\chi(\mathcal{M}_4(\tilde{\ell})),
\end{eqnarray*}
while
\begin{eqnarray*}
 \chi(\mathcal{M}_6(\ell))&=&5 \,\,\,\not=\,\,\, \chi(\mathcal{M}_6(\tilde{\ell})).
\end{eqnarray*}
We can still show that $\mathcal{M}_4(\ell)$ does not have the same homotopy type as $\mathcal{M}_4(\tilde{\ell})$. To see this, note that in the filtration $(\mathcal{M}^j)$ the relative complex $(X^4_4,\partial X^4_4)$ is attached $6=a_0(\ell)+a_1(\ell)$ times, so that before attaching the final $(X^5_4,\partial X^5_4)$, we have $H_{11}(\mathcal{M}^{m-1};\Q) \cong \Q^6$. Since $H_{12}(X^5_4,\partial X^5_4;\Q)\cong \Q$, we get that the $11$-th Betti number of $\mathcal{M}_4(\ell)$ is at least $5$. As $H_{11}(\mathcal{M}_4(\tilde{\ell}))\cong \Z/2\Z$ by \cite[Table 5.3]{kebacl}, these spaces have different homology.
\end{example}

This last argument can be generalized to obtain Morse-type inequalities; we only give a few special cases.

\begin{proposition}
 Let $\ell\in \R^n$ be a generic length vector. For $n\geq 6$ we have
\begin{eqnarray*}
 b_{3(n-1)-10}(\mathcal{M}_4(\ell))-b_{3(n-1)-9}(\mathcal{M}_4(\ell)) & = & a_1(\ell)-a_{n-3}(\ell)-a_0(\ell),
\end{eqnarray*}
and for $n\geq 7$ we have
\begin{eqnarray*}
 c_2(\ell)+c_1(\ell)+2 c_0(\ell) \,\,\,\, \geq \,\,\,\, b_{3(n-2)-10}(\mathcal{M}_4(\ell)) & \geq & c_2(\ell)-c_1(\ell),
\end{eqnarray*}
where $c_i(\ell)= a_i(\ell)-a_{n-2-i}(\ell)$ for $i=0,1,2$.

Also, for $n\geq 9$ we have
\begin{eqnarray*}
 b_{5(n-1)-21}(\mathcal{M}_6(\ell))-b_{5(n-1)-20}(\mathcal{M}_6(\ell)) & = & a_1(\ell)-a_{n-3}(\ell)-2a_0(\ell),
\end{eqnarray*}
and for $n\geq 10$ we have
\begin{eqnarray*}
 c_2(\ell)+c_1(\ell)+c_0(\ell) \,\,\,\, \geq \,\,\,\, b_{5(n-2)-21}(\mathcal{M}_6(\ell)) & \geq & c_2(\ell)-c_1(\ell)-c_0(\ell).
\end{eqnarray*}

\end{proposition}

\begin{proof}
 The dimension of $X^{n-4}_4$ is $3(n-1)-10$, and in the filtration arising from the standard Morse-Bott function we get $c_1+c_0$-many of those. Therefore the $(3n-1)-10$-th Betti number is at most $c_1+c_0$. Furthermore, only one $X^{n-3}_4$ can occur in the filtration, and only at the very end, so the $3(n-1)-9$-th Betti number can be at most $a_0(\ell)$. When obtaining the homology of $\mathcal{M}_4(\ell)$ from the filtration, this generator in degree $(3(n-1)-9$ may or may not cancel with a generator in degree $3(n-1)-10$. In either case, the difference of the Betti numbers is as claimed.

To determine $b_{3(n-2)-10}(\mathcal{M}_4(\ell))$, note that only $X^{n-5}_4$'s and the final $X^{n-3}_4$ can contribute to this number, and each $X^{n-4}_4$ may cancel a generator. As $(X^{n-4}_4,\partial X^{n-4}_4)$ has homology in degrees $3(n-2)-11$ and $3(n-2)-9$, we could get two cancellations. As there are $(c_2+c_1+2c_0)$-many $X^{n-5}_4$ and $X^{n-3}_4$, and $(c_1+c_0)$-many of $X^{n-4}_4$ in the filtration, the result follows.

The result for $d=6$ is analogous, but we use a different Poincar\'e polynomial for $(X^{n-4}_6,\partial X^{n-4}_6)$, see Corollary \ref{someformulas}. Note that the conditions $n\geq 9$ and $n\geq 10$ ensure that $Q_2(t^4)$ is used, which means that the coefficient of $t^4$ is $2$. If $Q_1(t^4)$ is used, the formulas slightly improve.
\end{proof}

\begin{appendix}
 
\section{Shape spaces}

In this appendix, we define the shape spaces occuring in statistical shape theory. Roughly speaking, a shape is a collection of points in some $\R^d$ up to rotations and translations and scalings. More information can be found for example in the book \cite{kebacl}.

We begin by defining the pre-shape space. Let

\begin{eqnarray*}
 \mathbf{S}^n_d&=&\left\{(x_1,\ldots,x_n)\in (\R^d)^n\,\left| \sum_{i=1}^n x_i=0,\, \sum_{i=1}^n|x_i|^2 = 1\right\}
\right..
\end{eqnarray*}
This is the space of $n$ points in $\R^d$, whose centroid is $0$ and who are scaled to sit on the unit sphere. Notice that $\mathbf{S}^n_d$ is the intersection of the sphere of dimension $nd-1$ with a sub-vector space of codimension $d$. It is therefore a sphere of dimension $(n-1)d-1$.

The group $SO(d)$ acts diagonally on $\mathbf{S}^n_d$, and the resulting quotient is called the \em shape space\em
\begin{eqnarray*}
 \Sigma^n_d&=& \mathbf{S}^n_d/SO(d).
\end{eqnarray*}

Similarly, one can define the \em size and shape space \em $S\Sigma^n_d$ which is obtained in the same way, but by dropping the condition that the points sit on the $(nd-1)$-sphere, see \cite[\S 11.2]{kebacl}.

As is pointed out in \cite{haurod}, if we define $T:S\Sigma^n_d\to \R^n$ by
\begin{eqnarray*}
 T(x_1,\ldots,x_n)&=&(|x_1-x_2|,\ldots,|x_{n-1}-x_n|,|x_n-x_1|)
\end{eqnarray*}
we get that $\mathcal{M}_d(\ell)=T^{-1}(\{\ell\})$. Furthermore, looking at inverse images of chambers and further stratas in $\R^d$ leads to a decomposition of the size and shape space by configuration spaces of linkages studied in \cite{haurod}.

An even more direct relation between linkage spaces and shape spaces is given by the next proposition.

\begin{proposition}
 \label{theshapespace}
Let $\ell=(1,\ldots,1,n-2)\in \R^n$. Then there exists an $SO(d-1)$-equivariant homeomorphism
\[
 \Phi:\mathcal{C}_d(\ell)\to \mathbf{S}^{n-1}_{d-1}.
\]
In particular, the shape space $\Sigma^{n-1}_{d-1}$ is homeomorphic to $\mathcal{M}_d(\ell)$.
\end{proposition}

\begin{proof}
 Let $(x_1,\ldots,x_{n-1})\in \mathcal{C}_d(\ell)$. If $p_1:\R^d\to \R$ is projection to the first coordinate, notice that $p_1(x_i)<0$ by elementary geometry. Now let $p:\R^d\to \R^{d-1}$ be projection to the last $d-1$ coordinates. Let
\begin{eqnarray*}
 c&=&\sum_{i=1}^{n-1}|p(x_i)|^2.
\end{eqnarray*}
Then $c>0$, for otherwise each $x_i=-e_1\in \R^d$, and $$x_1+\cdots+x_{n-1}=(-n+1,0,\ldots,0)\not=(-n+2,0,\ldots,0).$$
So we can define
\begin{eqnarray*}
 \Phi(x_1,\ldots,x_{n-1})&=&(p(x_1)/\sqrt{c},\ldots,p(x_{n-1})/\sqrt{c}).
\end{eqnarray*}
This is injective: If $\Phi(x_1,\ldots,x_{n-1})=\Phi(y_1,\ldots,y_{n-1})$, then $p(x_i)=ep(y_i)$ for some $e>0$ and all $i=1,\ldots,n-1$. If $e>1$, then $p_1(x_i)<p_1(y_i)$ for all $i=1,\ldots,n-1$, but then
\begin{eqnarray*}
 \sum_{i=1}^{n-1} p_1(x_i)&\not= & \sum_{i=1}^{n-1} p_1(y_i),
\end{eqnarray*}
contradicting that both are in $\mathcal{C}_d(\ell)$. The case $e<1$ leads to a similar contradiction, and $e=1$ implies $(x_1,\ldots,x_{n-1})=(y_1,\ldots,y_{n-1})$.

As $\mathcal{C}_d(\ell)$ is a closed manifold of the same dimension, $\Phi$ is also surjective. Equivariance is clear from the construction, so the statement follows.
\end{proof}

\section{Polynomial relations}

We want to collect a few properties of the sequence of polynomials $Q_n(t)$ given by
\begin{eqnarray*}
 Q_{2m}(t)&=&\frac{(t^{m+1}-1)^2}{(t-1)^2}\\
Q_{2m+1}(t)&=& \frac{(t^{m+2}-1)(t^{m+1}-1)}{(t-1)^2}
\end{eqnarray*}
for all $m\geq 0$. For convenience, we also add the equations $Q_n(t)=0$ for $n<0$.

The next lemma follows directly from the fact that
\begin{eqnarray*}
 \frac{t^{m+1}-1}{t-1}&=& 1+ t + \cdots + t^m.
\end{eqnarray*}

\begin{lemma}\label{qiterate}
 Let $m\geq 1$, then
\begin{eqnarray*}
 Q_{2m-1}(t)&=& Q_{2m-2}(t)+t^m+\cdots + t^{2m-1}\\
 Q_{2m}(t)&=& Q_{2m-1}(t) + t^m+\cdots + t^{2m}
\end{eqnarray*}

\end{lemma}

Using an induction on Lemma \ref{qiterate}, we get a nice description for the coefficients of $Q_n(t)$.

\begin{lemma}
 For $m\geq 0$ we have
\begin{eqnarray*}
 Q_{2m}(t)&=&1 + 2t + \cdots + (m+1)t^m + m t^{m+1} + \cdots + 2 t^{2m-1} + t^{2m}\\
 Q_{2m+1}(t)&=& 1 + 2t + \cdots + (m+1)t^m + (m+1)t^{m+1} + \cdots + 2 t^{2m} + t^{2m+1}.
\end{eqnarray*}

\end{lemma}

The next lemmata also follow by induction using Lemma \ref{qiterate}.

\begin{lemma}\label{qsumformulas}
 For $m\geq 0$ we have
\begin{eqnarray*}
 Q_m(t)&=& \sum_{i=0}^{\lfloor\frac{m}{2}\rfloor} t^i \, \frac{t^{m+1-2i}-1}{t-1}\\
&=& \sum_{i=0}^m t^{\lfloor\frac{i+1}{2}\rfloor} \, \frac{t^{\lfloor\frac{i+2}{2}\rfloor}-1}{t-1}.
\end{eqnarray*}

\end{lemma}

\begin{lemma}\label{qsumformulasx}
 For $m\geq 0$ we have
\begin{eqnarray*}
 Q_m(t^2)+tQ_{m-1}(t^2) &=& \sum_{i=0}^m t^i \,  \frac{t^{2(n+1-i)}-1}{t^2-1}
\end{eqnarray*}

\end{lemma}

\end{appendix}

\end{document}